\documentclass[a4paper,12pt,reqno]{amsart}
\usepackage{amssymb,amsthm,amsmath,amsfonts}
\usepackage[numbers]{natbib}
\usepackage[colorlinks,citecolor=blue,urlcolor=blue]{hyperref}
\usepackage{hypernat}
\begin{document}

\newtheorem{theorem}{Theorem}
\newtheorem{proposition}{Proposition}
\newtheorem{lemma}{Lemma}
\theoremstyle{remark}
\newtheorem*{rem}{Remark}

\def\U{\mathaccent'27U}
\def\Prob{\mathsf P}
\def\sh{\sinh}
\def\ch{\cosh}

\def\MR#1{\href{http://www.ams.org/mathscinet-getitem?mr=#1}{MR#1}}

\title{Exact $L_2$-small deviation asymptotics\\ for some Brownian functionals}

\author{Ya.~Yu.~NIKITIN}
\author{R.~S.~PUSEV}
\address{Department of Mathematics and Mechanics\\ Saint-Petersburg State University}
\email{yanikit47@gmail.com, Ruslan.Pusev@math.spbu.ru}

\begin{abstract}
We find exact small deviation asymptotics  with respect to weighted
Hilbert norm for some well-known Gaussian processes. Our approach does not require the knowledge of eigenfunctions
of the covariance operator of a weighted process.  Such a peculiarity of the method  makes it possible
to generalize many previous results in this area. We also obtain new relations connected to exact small deviation
asymptotics for  a Brownian excursion, a Brownian meander, and Bessel processes and bridges.
\end{abstract}



\maketitle

\section{Introduction}

The theory of small deviations for Gaussian processes has developed rapidly  over the recent years (see the surveys  \cite{Li:Shao:2001},  \cite{Lifs:1999}, \cite{Fata:2003}, and the complete bibliography  in \cite{Lifs:2010}.)
Such a development was and is stimulated by numerous links between the small deviation theory and some important mathematical problems, including  the accuracy of discrete approximation for random processes, the calculation of the metric entropy for functional sets, the law of the iterated logarithm in the Chung  form, and the quantization problem. It has been also observed that the small deviation theory is related to the functional data analysis \cite{Ferr:Vieu:2006} and nonparametric Bayesian estimation  \cite{Aurz:Ibra:Lifs:vZan:2008},  \cite{vdVa:vZan:2008}.

The problem of small deviations (also called the small ball problem) of a random process $X$ in the norm $\|\cdot\|$ consists of describing the behavior of the probability $\Prob\{\|X\|\leq\varepsilon\}$ as $\varepsilon\to0$.  An asymptotic relation of the type
$$
\Prob\{\|X\|\leq\varepsilon\}\sim
C\varepsilon^\beta\exp(-d\varepsilon^{-\alpha}),
\quad
\varepsilon\to0,
$$
with some real constants $C,\ \beta, \ d $ and $\alpha$, is referred to as exact asymptotics. A less precise statement of the form
$$
\log\Prob\{\|X\|\leq\varepsilon\}\sim
-d\varepsilon^{-\alpha},
\quad
\varepsilon\to0,
$$
is known as logarithmic asymptotics.

It is noted in the well-known monograph by Lifshits \cite[Sec. 18]{Lifs:1995} that: ``Unlike the large deviations, the behavior of small deviations cannot be uniformly described for the whole class of Gaussian measures, even on the logarithmic level. The formalism for estimating the small deviations, which would be as simple as application of the action functional for large deviations, has not been discovered yet. Only partial results are known for several significant particular situations...''

In the literature on small deviations, it is rare, and only for a limited number of random processes, that  exact and logarithmic asymptotics with sharp constants are found \cite{Li:Shao:2001}, \cite{Fata:2003}.
In this paper, we are focused on the task of deriving the {\it exact } small deviation asymptotics  using Hilbert norm.

 Let  $X(t)$, $a\leq t\leq b$, be a Gaussian process with zero mean and covariance function $G(t,s)$, $a\leq t,s\leq b.$
  For a non-negative weight function  $\psi(t)$ defined on  $[a,b]$, put
$$
\|X\|_\psi=
\left(
\int_a^b X^2(t)\psi(t)dt
\right)^{1/2}   .
$$

If the integral $\int_a^b G(t,t)\psi(t)dt$ is finite, then the process $X(t)\sqrt{\psi(t)}$ admits the Karhu\-nen-Lo\`{e}ve expansion  (see, e.g., \cite{Adle:1990}):
$$
X(t)\sqrt{\psi(t)}=\sum_{k=1}^\infty\xi_k\sqrt{\lambda_k}f_k(t),
$$
where $\xi_k$, $k\in\mathbb N,$  are independent standard normal random variables, while  $\lambda_k>0$  and $f_k(t)$, $k\in\mathbb N$, are the eigenvalues and eigenfunctions of the Fredholm integral equation
$$
\lambda f(t)=\int_a^b G(t,s)\sqrt{\psi(t)\psi(s)}f(s)\,ds,\quad t\in[a,b].
$$

It follows from the Karhunen-Lo\`{e}ve expansion that
$$
\|X\|_\psi^2  = \int_a^bX^2(t)\psi(t)\,dt  \stackrel{law}{=} \sum_{k=1}^\infty\lambda_k\xi_k^2.
$$

Thus, the initial small deviation problem is reduced to the asymptotic analysis of 
$\Prob\left\{\sum_{k=1}^\infty\lambda_k\xi_k^2\leq\varepsilon^2\right\}$ as $\varepsilon\to0$.
First solutions of this problem were based on the evaluation of Laplace transform and its subsequent inversion, see, e.g., the pioneer paper by Sytaya \cite{Syta:1974} and  surveys \cite{Li:Shao:2001} and  \cite{Lifs:1999}. An intricate form of the solution in \cite{Syta:1974} makes it impossible to extract the exact asymptotics from it.
Starting from  \cite{Ibra:1979} and \cite{Dudl:Hoff:Shep:1979}, many authors attempted simplifying the  asymptotic expression for the small deviation probability under various conditions.

Zolotarev \cite{Zolo:1979} obtained the exact small deviation asymptotics in the case $\lambda_k=k^{-A}$, $A>1$. Using the results of \cite{Lifs:1997},
Dunker, Lifshits and Linde  \cite{DLL}  found the exact asymptotics  when $\lambda_k=f(k)$, where $f$ is a positive, logarithmically convex and twice differentiable summable function. In  \cite{Begh:Niki:Orsi:2003} the results of \cite{DLL} were applied to the case of integrated and centered (by time)  Brownian motion and Brownian bridge. More general results for $m$-times integrated processes were later proved in  \cite{Gao:Hann:Lee:Torc:2004} and  \cite{Naza:Niki:2004}. For the Slepian process the precise solution to the small ball problem was described in \cite{Niki:Orsi:2004}. In \cite{Niki:Khar:2004} the same problem was solved for the important special case when the eigenvalues $\lambda_k$ are the ratios of  powers of two polynomials with real coefficients.

 A new approach enabling the analysis of small deviation asymptotics in  $L_2$-norm, up to a constant, has been elaborated in \cite{Naza:Niki:2004}
 and \cite{Naza:2009}. The main results of \cite{Naza:Niki:2004}
 and \cite{Naza:2009} are applied for a large class of Gaussian processes whose covariances coincide with Green functions of self-adjoint differential operators of rather general form.  We suggest to call such processes ``Green processes''.

 The small deviation asymptotics  for {\it weighted}  Green processes  was established by Na\-za\-rov and Pusev \cite{Naza:Puse:2009}, who
  obtained, for sufficiently smooth non-degenerate weight functions, the small deviation asymptotics up to the  so-called  {\it distortion constant.} This constant appeared for the first time in the paper of Li  \cite{Li:1992}  and  has the form of a certain infinite product.  Evaluating this constant  requires the knowledge of eigenfunctions of the covariance. Using the approach proposed in \cite{Naza:2003}, see also similar results in \cite{Gao:Hann:Lee:Torc:2003},  Nazarov and Pusev  \cite{Naza:Puse:2009} calculated the distortion constant for certain weighted Green processes with known eigenfunctions.   For the reader's convenience, in Section 2 we formulate the corresponding theorem along with some auxiliary results. The range of processes that satisfy conditions of the main theorem in Section 2 is wide and includes, for example, the Brownian motion, the Brownian bridge, the Ornstein-Uhlenbeck process and their multiple integrated analogues.

 In this paper we show how to evaluate the distortion constant for some Gaussian processes when the {\it eigenfunctions of the covariance are unknown.}   In
  Section 3 we study the small deviation probabilities of some weighted Gaussian processes with unknown eigenfunctions. Using the approach closely related to the classical WKB method \cite{From:From:1967}, we first find asymptotics of the eigenfunctions and then derive the distortion constant. This gives the desired exact small deviation asymptotics.
  For the weighted Hilbert norm, the method that we use simplifies considerably the evaluation of  exact asymptotics. We expect that the proposed approach is applicable for a much more general subclass of Green processes than the one considered in Section 3.

In Section 4, using the results of Section 3, we study the exact small deviation asymptotics in the weighted  $L_2$-norm for Bessel processes and Bessel bridges. In Section 5, the results of Section 4 are applied to the Brownian excursion.
Then, in Sections 6 and 7, we explore analogous problems for Brownian local times, the Brownian meander, and similar processes.
In particular, we consider integral functionals of Bessel processes and Brownian local times, the supremum process of the Brownian motion  and the Pitman process, the suprema of Bessel processes, of the Brownian excursion and Brownian meander. The results are new even under the unit weight, though their proofs are simple and rely on certain known identities in law between Brownian functionals.

We expect that the appearance of  tables of exact small deviation asymptotics for various functionals of random processes
is just a matter of time. Such tables should be similar to the tables of integrals, sums and products or to the tables of distributions of functionals of Brownian motion.

\section{Auxiliary statements}

Let $L$  be the self-adjoint linear differential operator of order $2\ell$,
defined on the space $\mathcal D(L)$ of functions satisfying  $2\ell$ boundary conditions.
Denote by $W_p^m (0,1)$ the Banach space of
$(m-1)$ times continuously differentiable  functions $y$ such that the derivative  $y^{(m-1)}$
is absolutely continuous on  $[0,1]$ and  $y^{(m)}\in L_p(0,1)$.  Next lemma was proved in
\cite[Lemma 2.1]{Naza:Puse:2009}.

\begin{lemma}
\label{mainlemma}
Consider a function $\psi\in W_\infty^\ell(0,1)$  such that $\psi>0$ on $(0,1)$.
Let  $G(t,s)$ be the Green function of the boundary-value problem
$$
Lv=\mu v \quad \text{on} \quad [0,1],\qquad v\in\mathcal D(L).
$$
Then the function
$$
\mathcal G(t,s)=\sqrt{\psi(t)\psi(s)}G(t,s)
$$
is the Green function of the boundary-value problem
\begin{equation}\label{Lu}
\mathcal Lv\equiv\psi^{-1/2}L(\psi^{-1/2}v)=\mu v
\quad \text{íà} \quad [0,1],\qquad v\in\mathcal D(\mathcal L),
\end{equation}
where the space $\mathcal D(\mathcal L)$ consists of functions $v$
which satisfy the condition
\begin{equation}\label{bc}
\psi^{-1/2}v\in\mathcal D(L).
\end{equation}
\end{lemma}

\begin{rem}
By setting $y=\psi^{-1/2}v$ we can rewrite the problem (\ref{Lu})--(\ref{bc}) as follows:
$$
Ly=\mu\psi y \quad \text{on} \quad [0,1],\qquad y\in\mathcal D(L).
$$
\end{rem}

In the same paper  \cite{Naza:Puse:2009},  the following theorem is proved by means of  Lemma \ref{mainlemma} and the theory developed in
\cite{Naza:Niki:2004}:
\begin{theorem}
\label{mainthm}
Let the covariance function~$G_X(t,s)$ of the centered Gaussian process~$X(t)$,
$0\leq t\leq1$, be the Green function of the self-adjoint differential operator~$L$ of order~$2\ell$
$$
Lv\equiv
(-1)^\ell v^{(2\ell)}+\left(p_{\ell-1}v^{(\ell-1)}\right)^{(\ell-1)}+\ldots+p_0v;
$$
$$p_m\in L_1(0,1),\quad m=0,\ldots,\ell-2;\qquad p_{\ell-1}\in L_\infty(0,1);$$
with boundary conditions
$$
\left.
\begin{aligned}
U_j^0(v)\equiv v^{(k_j)}(0)+\sum_{k<k_j}\alpha_{jk}^0v^{(k)}(0)=0,\\
U_j^1(v)\equiv v^{(k'_j)}(1)+\sum_{k<k'_j}\alpha_{jk}^1v^{(k)}(1)=0,
\end{aligned}
\right\}
j=1,\ldots,\ell,
$$
where $\alpha_{jk}^i$ are some constants,
$$0\leq k_1<\ldots<k_\ell\leq2\ell-1,\quad
0\leq k'_1<\ldots<k'_\ell\leq2\ell-1.$$
Assume that
$$\varkappa\equiv\sum\limits_{j=1}^\ell(k_j+k'_j)<2\ell^2.$$
Let $\psi\in W_\infty^\ell(0,1)$ and $\psi(x)>0$, $x\in[0,1]$.
Then as $\varepsilon\to0$
$$
\Prob(\|X\|_\psi\leq\varepsilon)\sim
\mathcal C\, \varepsilon^\gamma
\exp\left(-\frac{2\ell-1}2 \left(\frac{\vartheta_\ell}{2\ell\sin\frac\pi{2\ell}}\right)
^{\frac{2\ell}{2\ell-1}}
\varepsilon^{-\frac2{2\ell-1}}\right),
$$
where
$$
\gamma=-\ell+\frac{\varkappa+1}{2\ell-1},\quad
\vartheta_\ell=\int_0^1\psi^{1/{(2\ell)}}(x)\,dx,$$
\begin{equation}
\mathcal C=C_{\rm dist}\,
\frac{(2\pi)^{\ell/2}(\pi/\vartheta_\ell)^{\ell\gamma}(\sin\frac\pi{2\ell})^{\frac{1+\gamma}2}}
{(2\ell-1)^{1/2}(\frac\pi{2\ell})^{1+\frac\gamma2}\Gamma^\ell(\ell-\frac\varkappa{2\ell})}.
\label{C(X)}
\end{equation}

In {\rm (\ref{C(X)})} the ``distortion constant'' $C_{\rm dist}$  is given by
$$C_{\rm dist}=\prod_{n=1}^\infty\frac{\mu_n^{1/2}}
{\left(\pi/\vartheta_\ell\cdot\left[n+\ell-1-\frac\varkappa{2\ell}\right]\right)^\ell},$$
where $\mu_n$ are the eigenvalues of the boundary-value problem
$$Ly=\mu\psi y,\quad
U_j^0(y)=0,\quad U_j^1(y)=0,\quad j=1,\ldots,\ell.$$
\end{theorem}

\medskip

Next, we need the lemma of M.~A.~Lifshits (see, e.g., \cite{Puse:2010}), which is proved by direct yet laborious calculations.

\begin{lemma}
\label{LifshitsLemma}
Let $V_1$, $V_2>0$ be two independent random variables with known small deviation asymptotics; namely, assume that as $r\to0$
$$
\Prob\{V_1\leq r\}\sim K_1r^{a_1}\exp\left(-D_1^{d+1}r^{-d}\right),
$$
$$
\Prob\{V_2\leq r\}\sim K_2r^{a_2}\exp\left(-D_2^{d+1}r^{-d}\right),
$$
where $a_1$, $a_2$ are arbitrary real constants, and $K_1$, $K_2$, $D_1$, $D_2$, $d$ are arbitrary positive numbers.
Then
$$
\Prob\{V_1+V_2\leq r\}\sim Kr^a\exp\left(-D^{d+1}r^{-d}\right),
$$
where
$$
D=D_1+D_2,
\quad
a=a_1+a_2-\frac{d}{2},
\quad
K=K_1K_2\sqrt{\frac{2\pi d}{d+1}}\cdot\frac{D_1^{a_1+\frac12}D_2^{a_2+\frac12}}{D^{a+\frac12}}.
$$
\end{lemma}

Using induction by $n$, it is straightforward to extend Lemma \ref{LifshitsLemma}
to the case of an arbitrary number of i.i.d. random variables. We have the following result.

\begin{lemma}
\label{LemmaLifs3}
Let $V_1$, \ldots, $V_n$ be i.i.d. positive  random variables such that
$$
\Prob\{V_i\leq r\}\sim Kr^{a}\exp\left(-D^{2}r^{-1}\right), \quad r \to 0,
$$
where $a$ is a real number, and $K$ and $D$ are positive constants.
Then
$$
\Prob\{V_1+\ldots+V_n\leq r\}\sim \widetilde Kr^{\widetilde a}\exp\left(-{\widetilde D}^{2}r^{-1}\right), \quad r \to 0,
$$
where
$$
\widetilde D=nD,
\quad
\widetilde a=na-\frac{n-1}{2},
\quad
\widetilde K=\frac{K^nD^{n-1}\pi^{(n-1)/2}}{n^{\widetilde a+1/2}}.
$$
\end{lemma}

\section{Small deviations in the weighted quadratic norm}

Define
$$
\vartheta=\int_0^1\sqrt{\psi(t)}dt.
$$
The next theorem gives the exact small deviation asymptotics for the Brownian bridge  $B$  in a weighted $L_2$-norm for a large class of weights.
\begin{theorem}
\label{weightedBridge}
Let function $\psi$ defined on $[0,1]$ be positive and twice continuously differentiable.  Then as $\varepsilon\to0$
\begin{equation}
\label{psiWeightedBridge}
\Prob\{\|B\|_\psi\leq\varepsilon\}\sim
\frac{2\sqrt{2}\psi^{1/8}(0)\psi^{1/8}(1)}{\sqrt{\pi\vartheta}}
\exp\left(-\frac{\vartheta^2}{8}\varepsilon^{-2}\right).
\end{equation}
\end{theorem}

\begin{proof}
According to Lemma \ref{mainlemma}, the coefficients $\lambda_k$ in the Karhunen-Lo\`{e}ve expansion are given by $\lambda_k=\mu_k^{-1}$, where  $\mu_k$~are the eigenvalues of the boundary-value problem
$$
\left\{
\begin{aligned}
&-y''=\mu\psi y\quad \text{on} \quad [0,1],\\
&y(0)=y(1)=0.\\
\end{aligned}
\right.
$$

Let $\varphi_{1}(t,\zeta)$ and $\varphi_{1}(t,\zeta)$ be solutions of  the equation $-y''=\zeta^2\psi y$ satisfying the initial conditions
\begin{align}
&\varphi_1(0,\zeta)=1,&\varphi_1'(0,\zeta)=0,\label{fs1}\\
&\varphi_2(0,\zeta)=0,&\varphi_2'(0,\zeta)=1.\label{fs2}
\end{align}
Such a choice of the fundamental system of solutions is convenient
from a technical point of view. It allows us to study the behavior of solutions for $\zeta$ in both neighborhoods, the neighborhood of zero and the neighborhood of infinity.

Upon substituting the general solution  $y(t)=c_1\varphi_1(t,\zeta)+c_2\varphi_2(t,\zeta)$ into the boundary conditions, we observe that $\mu_k=x_k^2$, where $x_1<x_2<\ldots$ are the positive roots of the function
$$
F(\zeta)=\det
\begin{bmatrix}
\varphi_1(0,\zeta)
&
\varphi_2(0,\zeta)
\\
\varphi_1(1,\zeta)
&
\varphi_2(1,\zeta)
\end{bmatrix}
=\varphi_2(1,\zeta).
$$

Due to Theorem \ref{mainthm}, it is sufficient to prove that
$$
C_{\rm dist}^2\equiv
\prod_{k=1}^\infty
\frac{x_k^2}{(\pi k/\vartheta)^2}
=\frac{\psi^{1/4}(0)\psi^{1/4}(1)}{\vartheta}.
$$

We shall calculate this infinite product by means of Jensen's theorem.  Let $f(\zeta)$~be a function of complex variable which is analytic for $|\zeta|<R$. Suppose that
$f(0)\neq0$, and let $r_1,r_2,\ldots$~be the modules of zeros of the function $f(\zeta)$ in the circle $|\zeta|<R$, arranged in the non-decreasing order. By Jensen's theorem  (see, e.g., \cite[\S 3.6.1]{Titch:80}) for
$r_k<r<r_{k+1}$
we have
$$
\ln \left( \frac{r^k|f(0)|}{r_1r_2\ldots r_k}\right) =\frac{1}{2\pi}\int_0^{2\pi}\ln|f(re^{i\theta})|\,d\theta.
$$

Therefore for two functions $f$ and $g$ with modules of zeros $r_1,r_2,\ldots$ and $s_1,s_2,\ldots$,
respectively,  for all $\max\{r_k,s_k\}<r<\min\{r_{k+1},s_{k+1}\}$ we get
$$
\ln \left( \frac{|f(0)|s_1s_2\ldots s_k}{|g(0)|r_1r_2\ldots r_k}\right)
=\frac{1}{2\pi}\int_0^{2\pi}\ln\frac{|f(re^{i\theta})|}{|g(re^{i\theta})|}\,d\theta.
$$
Consequently,
$$
\prod_{k=1}^\infty\frac{r_k}{s_k}=\left|\frac{f(0)}{g(0)}\right|
\lim_{r\to\infty}\exp\left\{\frac{1}{2\pi}
\int_0^{2\pi}\ln\frac{|g(re^{i\theta})|}{|f(re^{i\theta})|}\,d\theta\right\}.
$$

In order to study the asymptotic behavior of the function $F(\zeta)$ as $|\zeta|\to\infty,$ we shall utilize the so-called WKB approximation,
which has been used long ago in quantum mechanics
for the analysis of the Schr\"{o}dinger equation. The origin of the WKB method dates back to old papers by Carlini, Green and Liouville, see
\cite{From:From:1967} for a history of the method.

According to \cite[Ch. 2, \S 3]{Fedo:2009}, the equation $-y''=\zeta^2\psi y$ has solutions of the form
\begin{equation}
\label{EFasympt}
\widetilde\varphi_{1,2}(t,\zeta)=
\psi^{-1/4}(t)\exp\left(\pm i\zeta\int_0^t\sqrt{\psi(u)}du\right)
\left(1+\frac{\delta_{1,2}(t,\zeta)}{\zeta}\right),
\end{equation}
where  the functions  $\delta_{1,2}(t,\zeta)$  satisfy
$$
|\delta_{1,2}(t,\zeta)|\leq C,
$$
uniformly in $t\in[0,1]$ and  $\zeta\in D_r^-=\{|\zeta|\geq r>0,\mathop{\rm Im}(\zeta)\leq0\}.$
Differentiating asymptotic expressions for the functions $\widetilde\varphi_{1,2}(t,\zeta),$ we obtain
\begin{equation}
\label{EFderasympt}
\widetilde\varphi_{1,2}^{\ '}(t,\zeta)=
\pm i\zeta\psi^{1/4}(t)\exp\left(\pm i\zeta\int_0^t\sqrt{\psi(u)}du\right)
\left(1+\frac{\widetilde\delta_{1,2}(t,\zeta)}{\zeta}\right),
\end{equation}
where the functions  $\widetilde\delta_{1,2}(t,\zeta)$ are also uniformly bounded.

Observe that Wronskian of the solutions $\widetilde\varphi_{1,2}(t,\zeta)$  does not vanish
when $|\zeta|$ gets large. Hence for sufficiently large  $|\zeta|$ the functions $\widetilde\varphi_{1,2}(t,\zeta)$ are linearly independent, so that 
for large $|\zeta|$ the functions $\varphi_{1,2}(t,\zeta)$  can be represented as linear combinations of  $\widetilde\varphi_{1,2}(t,\zeta)$:
$$
\varphi_1(t,\zeta)=c_{11}(\zeta)\widetilde\varphi_1(t,\zeta)+c_{12}(\zeta)\widetilde\varphi_2(t,\zeta),
$$
$$
\varphi_2(t,\zeta)=c_{21}(\zeta)\widetilde\varphi_1(t,\zeta)+c_{22}(\zeta)\widetilde\varphi_2(t,\zeta).
$$
Due to conditions (\ref{fs1})--(\ref{fs2}), we have
$$
c_{11}(\zeta)=\frac{\psi^{1/4}(0)}{2}+O(\zeta^{-1}),
\quad
c_{12}(\zeta)=\frac{\psi^{1/4}(0)}{2}+O(\zeta^{-1}),
$$
$$
c_{21}(\zeta)=\frac{1}{2i \zeta\psi^{1/4}(0)}+O(\zeta^{-2}),
\quad
c_{22}(\zeta)=-\frac{1}{2i \zeta\psi^{1/4}(0)}+O(\zeta^{-2}).
$$
Thus, the following asymptotic relations hold for $|\zeta|\to\infty$, $\mathop{\rm Im}(\zeta)\leq0$:
\begin{equation}
\label{phi1}
\varphi_1(1,\zeta)=\frac{\psi^{1/4}(0)\cos(\vartheta\zeta)}{\psi^{1/4}(1)}(1+O(\zeta^{-1})),
\end{equation}
\begin{equation}
\varphi'_1(1,\zeta)=-\psi^{1/4}(0)\psi^{1/4}(1)\zeta\sin(\vartheta\zeta)(1+O(\zeta^{-1})),
\end{equation}
\begin{equation}
\varphi_2(1,\zeta)=\frac{\sin(\vartheta\zeta)}{\psi^{1/4}(0)\psi^{1/4}(1)\zeta}(1+O(\zeta^{-1})),
\end{equation}
\begin{equation}
\label{phi4}
\varphi'_2(1,\zeta)=\frac{\psi^{1/4}(1)\cos(\vartheta\zeta)}{\psi^{1/4}(0)}(1+O(\zeta^{-1})).
\end{equation}
In a similar way, we can show that alike relations hold true as $|\zeta|\to\infty$, $\mathop{\rm Im}(\zeta)\geq0.$
So, we get as $|\zeta|\to\infty$:
$$
F(\zeta)=\varphi_2(1,\zeta)=\frac{\sin(\vartheta\zeta)}{\psi^{1/4}(0)\psi^{1/4}(1)\zeta}(1+O(\zeta^{-1}))
$$

Put $\Psi(\zeta)=\frac{\sin(\vartheta\zeta)}{\vartheta\zeta}$.
For large $k$ in the circle $|\zeta|<\frac\pi\vartheta(k+\frac12)$ there exist exactly $2k$ zeros $\pm\frac\pi\vartheta$, $\pm\frac{2\pi}\vartheta$, \ldots, $\pm\frac{k\pi}{\vartheta}$ of the function $\Psi(\zeta),$ and exactly $2k$ zeros $\pm \ x_j$, $j=1,\ldots,k$, of the function $F(\zeta)$.
Applying Jensen's theorem to the functions $F(\zeta)$ and $\Psi(\zeta),$ we obtain
$$
C_{\rm dist}^2=
\frac{\psi^{1/4}(0)\psi^{1/4}(1)|F(0)|}{\vartheta}.
$$
By the fact of continuous dependence of the solution of differential equation on a parameter (see, e.g., \cite[Ch.1, Sec. 7]{Cod:Lev:1955}), it follows that
$$
F(0)=\lim_{\zeta\to0}\varphi_2(1,\zeta)=\lim_{t\to1}\varphi_2(t,0).
$$
Obviously $\varphi_2(t,0)=t$. Therefore $F(0)=1.$
\end{proof}

\begin{rem}
  Choosing $\psi\equiv1$ and $\psi(t)=\exp(qt)$ in relation (\ref{psiWeightedBridge}) leads to the classical result on small deviation asymptotics for the Brownian bridge and to formula (3.16) of \cite{Naza:2003}, respectively.
\end{rem}

In the rest of the paper, the weight $\psi$ that satisfies the conditions of Theorem 2 will be called ``regular-shaped''.

\smallskip

By using the same arguments as above,  we can find the exact small deviation asymptotics for the
``elongated'' Brownian bridge $W_{(u)}(t)\equiv W(t)-utW(1)$. This is a centered Gaussian process with covariance function $G_{W_{(u)}}(t,s)=s\land t-(2u-u^2)st$, so that for $u\in(0,1]$ the process $W_{(u)}$ equals in law the Brownian bridge from zero to zero of length $(2u-u^2)^{-1}$ on the interval $[0,1].$  For $u=1$ this process coincides with a standard Brownian bridge. For $u=0$ it is a standard Brownian motion.

\begin{theorem}
\label{weightedW}
Consider the process $W_{(u)}(t)$ with $u<1$, and assume that $\psi$ is a regular-shaped weight on $[0,1]$. Then as $\varepsilon\to0$ 
\begin{equation}
\label{weightedWu}
\Prob\{\|W_{(u)}\|_\psi\leq\varepsilon\}\sim
\frac{4\psi^{1/8}(0)}{(1-u)\sqrt\pi\vartheta\psi^{1/8}(1)}
\varepsilon
\exp\left(-\frac{\vartheta^2}{8}\varepsilon^{-2}\right).
\end{equation}
\end{theorem}

\begin{proof}
By Lemma \ref{mainlemma} the numbers $\lambda_k$ in the Karhunen-Lo\`{e}ve series are equal to $\lambda_k=\mu_k^{-1}$, where  $\mu_k$~are the eigenvalues of the
boundary-value problem
$$
\left\{
\begin{aligned}
&-y''=\mu\psi y\quad \text{íà} \quad [0,1],\\
&y(0)=(y'+\tau y)(1)=0,\\
\end{aligned}
\right.
$$
where $\tau=(1-u)^{-2}-1$.

Denote by $\varphi_{1,2}(t,\zeta)$ solutions of the equation $-y''=\zeta^2\psi y$ satisfying the initial conditions (\ref{fs1})--(\ref{fs2}).

Substituting the general solution of equation $y(t)=c_1\varphi_1(t,\zeta)+c_2\varphi_2(t,\zeta)$ into the boundary conditions, we obtain
$\mu_k=x_k^2$, where $x_1<x_2<\ldots$ are positive zeros of the function
$$
F(\zeta)=\det
\begin{bmatrix}
\varphi_1(0,\zeta)
&
\varphi_2(0,\zeta)
\\
\varphi'_1(1,\zeta)+\tau\varphi_1(1,\zeta)
&
\varphi'_2(1,\zeta)+\tau\varphi_2(1,\zeta)
\end{bmatrix}
=\varphi'_2(1,\zeta)+\tau\varphi_2(1,\zeta).
$$
Taking into account Theorem \ref{mainthm}, it is sufficient to show that
$$
C_{\rm dist}^2\equiv
\prod_{k=1}^\infty
\frac{x_k^2}{(\pi (k-1/2)/\vartheta)^2}
=\frac{\psi^{1/4}(0)}{(1-u)^2\psi^{1/4}(1)}.
$$

It follows from relations (\ref{phi1})--(\ref{phi4}) that as $|\zeta|\to\infty$
$$
F(\zeta)=\frac{\psi^{1/4}(1)\cos(\vartheta\zeta)}{\psi^{1/4}(0)}(1+O(\zeta^{-1})).
$$
Applying Jensen's theorem to the functions $F(\zeta)$ and $\Psi(\zeta)=\cos(\vartheta\zeta),$ we get
$$
C_{\rm dist}^2=\frac{\psi^{1/4}(0)|F(0)|}{\psi^{1/4}(1)}.
$$
By the fact of continuous dependence of the solution of differential equation on a parameter, $F(0)=\lim\limits_{t\to1}(\varphi'_2(t,0)+\tau\varphi_2(t,0))$.
Simple calculations show that  $\varphi'_2(t,0)+\tau\varphi_2(t,0)=1+\tau t$, and hence $F(0)=1+\tau=(1-u)^{-2}$.
\end{proof}

\begin{rem}
For $u=0$ and $\psi\equiv1$ relation (\ref{weightedWu}) is a classical one. For $\psi(t)=\exp(qt)$ relation (\ref{weightedWu}) yields formula (3.15) of \cite{Naza:2003}.
Theorems \ref{weightedBridge} and \ref{weightedW} imply  Theorems 3.1 -- 3.4 and  Theorem 4.1 of \cite{Naza:Puse:2009}.
\end{rem}

Two subsequent theorems provide the exact small deviation asymptotics for the Orn\-stein-Uhlen\-beck process starting at zero and the usual stationary Ornstein-Uhlen\-beck process.  The Ornstein-Uhlenbeck process starting at zero, $\U_{(\alpha)}(t),$  is the centered Gaus\-sian process with the covariance function $$G_{\U_{(\alpha)}}(t,s)=(e^{-\alpha|t-s|}-e^{-\alpha(t+s)})/(2\alpha).$$

\begin{theorem}
\label{weightedOU}
Let $\psi$ be a regular-shaped weight on $[0,1]$. Then as $\varepsilon\to0$
\begin{equation}
\label{psiWeightedOU}
\Prob\{\|\U_{(\alpha)}\|_\psi\leq\varepsilon\}\sim
\frac{4e^{\alpha/2}\psi^{1/8}(0)}{\sqrt{\pi}\vartheta\psi^{1/8}(1)}
\varepsilon
\exp\left(-\frac{\vartheta^2}{8}\varepsilon^{-2}\right).
\end{equation}
\end{theorem}

\begin{proof}
Due to Lemma \ref{mainlemma}, $\lambda_k=\mu_k^{-1}$, where  $\mu_k$~are  eigenvalues of the problem
$$
\left\{
\begin{aligned}
&y''+(\mu\psi-\alpha^2)y=0\quad \text{on} \quad [0,1],\\
&y(0)=(y'+\alpha y)(1)=0.\\
\end{aligned}
\right.
$$

As before, denote by $\varphi_{1,2}(t,\zeta)$ the solutions of equation $y''+(\zeta^2\psi-\alpha^2)y=0$
which satisfy the initial conditions (\ref{fs1})--(\ref{fs2}).

On substituting the general solution $y(t)=c_1\varphi_1(t,\zeta)+c_2\varphi_2(t,\zeta)$ into the boundary conditions, we get  $\mu_k=x_k^2$, where $x_1<x_2<\ldots$ are  positive zeros of the function
$$
F(\zeta)=\det
\begin{bmatrix}
\varphi_1(0,\zeta)
&
\varphi_2(0,\zeta)
\\
\varphi'_1(1,\zeta)+\alpha\varphi_1(1,\zeta)
&
\varphi'_2(1,\zeta)+\alpha\varphi_2(1,\zeta)
\end{bmatrix}
=\varphi'_2(1,\zeta)+\alpha\varphi_2(1,\zeta).
$$

In view of Theorem \ref{mainthm}, it remains to show that
$$
C_{\rm dist}^2\equiv
\prod_{k=1}^\infty
\frac{x_k^2}{(\pi (k-1/2)/\vartheta)^2}
=\frac{e^\alpha\psi^{1/4}(0)}{\psi^{1/4}(1)}.
$$

In accordance with \cite[Ch. 2, Sec. 3]{Fedo:2009}, the differential equation $y''+(\zeta^2\psi-\alpha^2)y=0$ has  solutions $\widetilde\varphi_{1,2}(t,\zeta)$  satisfying  relations  (\ref{EFasympt})--(\ref{EFderasympt}). As in the proof of Theorem \ref{weightedBridge},  one can show that relations (\ref{phi1})--(\ref{phi4}) are valid for $\varphi_{1,2}(1,\zeta)$ and  $\varphi'_{1,2}(1,\zeta).$ Hence, as $|\zeta|\to\infty$
$$
F(\zeta)=\frac{\psi^{1/4}(1)\cos(\vartheta\zeta)}{\psi^{1/4}(0)}(1+O(\zeta^{-1})).
$$

Applying Jensen's theorem to the functions $F(\zeta)$ and $\Psi(\zeta)=\cos(\vartheta\zeta)$, we have
$$
C_{\rm dist}^2=
\frac{\psi^{1/4}(0)|F(0)|}{\psi^{1/4}(1)}.
$$
By continuity of the dependence of a solution of differential equation on a parameter, we get
$$
F(0)=\lim_{t\to1}(\varphi'_2(t,0)+\alpha\varphi_2(t,0)).
$$
Clearly $\varphi_2(t,0)=\frac1\alpha\sh(\alpha t)$,  and hence $F(0)=\ch\alpha+\sh\alpha=e^\alpha$.

\end{proof}

\begin{rem}

Theorem 4 extends the result of \cite[Corr. 3]{Gao:Hann:Lee:Torc:2003} obtained for the unit weight.
 It also extends Theorem 4.2 of \cite{Naza:Puse:2009} which corresponds to choosing $\psi(t)=\exp(qt)$ in (\ref{psiWeightedOU}).

\end{rem}

Denote by $U_{(\alpha)}(t)$ a usual stationary Ornstein-Uhlenbeck process, that is, a centered Gaussian process with  covariance $G_{U_{(\alpha)}}(t,s)=e^{-\alpha|t-s|}/(2\alpha)$.

\begin{theorem}
Consider the process $U_{(\alpha)}(t)$,  $\alpha>0$,  and assume that $\psi$ is a regular-shaped weight on $[0,1]$. Then as $\varepsilon\to0$
\begin{equation}
\label{psiWeightedOU2}
\Prob\{\|U_{(\alpha)}\|_\psi\leq\varepsilon\}\sim
\frac{8\alpha^{1/2}e^{\alpha/2}}{\pi^{1/2}\vartheta^{3/2}\psi^{1/8}(0)\psi^{1/8}(1)}
\varepsilon^2
\exp\left(-\frac{\vartheta^2}{8}\varepsilon^{-2}\right).
\end{equation}
\end{theorem}

\begin{proof}
By Lemma  \ref{mainlemma} we have $\lambda_k=\mu_k^{-1}$, where  $\mu_k$~are  eigenvalues of the problem
$$
\left\{
\begin{aligned}
&y''+(\mu\psi-\alpha^2)y=0\quad \text{on} \quad [0,1],\\
&(y'-\alpha y)(0)=(y'+\alpha y)(1)=0.\\
\end{aligned}
\right.
$$

Denote by $\varphi_{1,2}(t,\zeta)$ the solutions of equation $y''+(\zeta^2\psi-\alpha^2)y=0$ which satisfy conditions (\ref{fs1})--(\ref{fs2}).

Substituting the general solution $y(t)=c_1\varphi_1(t,\zeta)+c_2\varphi_2(t,\zeta)$ into the boundary conditions yields $\mu_k=x_k^2$, where $x_1<x_2<\ldots$ are  positive zeros of the function
\begin{align*}
F(\zeta)
&=\det
\begin{bmatrix}
\varphi'_1(0,\zeta)-\alpha\varphi_1(0,\zeta)
&
\varphi'_2(0,\zeta)-\alpha\varphi_2(0,\zeta)
\\
\varphi'_1(1,\zeta)+\alpha\varphi_1(1,\zeta)
&
\varphi'_2(1,\zeta)+\alpha\varphi_2(1,\zeta)
\end{bmatrix}=
\\
&=-\alpha\varphi'_2(1,\zeta)-\alpha^2\varphi_2(1,\zeta)-\varphi'_1(1,\zeta)-\alpha\varphi_1(1,\zeta).
\end{align*}

Acting as in the proof of Theorem 4.3 of \cite{Naza:Puse:2009}, we obtain
$$
\Prob\{\|U_{(\alpha)}\|_\psi\leq\varepsilon\}\sim
C_{\rm dist}
\frac{4\sqrt2}{\sqrt\pi\vartheta^2}
\varepsilon^2
\exp\left(-\frac{\vartheta^2}{8}\varepsilon^{-2}\right),
$$
where
$$
C_{\rm dist}^2=
\vartheta^2x_1^2
\prod_{k=2}^\infty
\frac{x_k^2}{(\pi (k-1)/\vartheta)^2}.
$$
Thus, it remains to show that
$$
C_{\rm dist}^2=\frac{2\alpha\vartheta e^\alpha}{\psi^{1/4}(0)\psi^{1/4}(1)}.
$$

The same arguments as in the proofs of the previous theorems lead to relations (\ref{phi1})--(\ref{phi4}) for  $\varphi_{1,2}(1,\zeta)$ and $\varphi'_{1,2}(1,\zeta).$
Therefore, as $|\zeta|\to\infty$
$$
F(\zeta)=\psi^{1/4}(0)\psi^{1/4}(1)\zeta\sin(\vartheta\zeta)(1+O(\zeta^{-1})).
$$
Denote $\Psi(\zeta)=((\vartheta\zeta)^2-1)\frac{\sin(\vartheta\zeta)}{\zeta}. $
Applying Jensen's theorem to the functions $F(\zeta)$ and $\Psi(\zeta),$ we get
$$
C_{\rm dist}^2=\frac{\vartheta|F(0)|}{\psi^{1/4}(0)\psi^{1/4}(1)}.
$$
By the fact of continuous dependence of the solution of differential equation on a parameter, we get
$$
F(0)=\lim_{t\to1}(-\alpha\varphi'_2(t,0)-\alpha^2\varphi_2(t,0)-\varphi'_1(t,0)-\alpha\varphi_1(t,0)).
$$
By direct calculations one has $\varphi_1(t,0)=\ch(\alpha t)$, $\varphi_2(t,0)=\frac1\alpha\sh(\alpha t)$.
Hence $|F(0)|=2\alpha e^\alpha$.

\end{proof}

\begin{rem}
Choosing $\psi(t)=\exp(qt)$ in (\ref{psiWeightedOU2}) leads to the statement of Theorem 4.3 of \cite{Naza:Puse:2009}. Also, Theorem  5  generalizes the well known results of \cite[Prop. 2.1]{Naza:2003}, \cite[Corr.3]{Gao:Hann:Lee:Torc:2003}, and \cite[p.140]{Fata:2008} obtained for the unit weight.
\end{rem}

In the next theorem we calculate the exact small deviation asymptotics for the Bogoliubov process  $Y(t)$, $t\in[0,1]$. This is a centered Gaussian process with  covariance
\begin{equation}
\label{BogCov} \mathsf EY(t)Y(s)= \frac{1}{2\omega\sh(\omega/2)}
\ch\left(\omega|t-s|-\frac{\omega}{2}\right), \quad t,s\in[0,1],
\quad \omega > 0.
\end{equation}
It is well known from the theory of Gaussian processes, see \cite[Example 7.1.5]{Boga:1997}, that the trajectories of the Bogoliubov process are continuous almost surely (a.s.). From (\ref{BogCov}) we get $\mathsf E(Y(1)-Y(0))^2=0$, so that $Y(0)=Y(1)$ a.s.  Hence, almost all trajectories of the Bogoliubov process belong to the space $C^0[0,1]$ of continuous functions $x(t)$ on $[0,1]$ satisfying the condition $x(0)=x(1)$. The distribution of the process $Y(t)$ in the space $C^0[0,1]$ endowed with
the uniform metric, is called the Bogoliubov measure $\mu_B.$ This measure was first considered in  \cite{Sank:1999,Sank:2001a}. It  plays an important role in the theory of statistical equilibrium of quantum systems and occurs in representing Gibbs equilibrium means of Bose operators in
the form of functional integrals using Bogoliubov's method of T-products \cite{Bogo:1954}.

Properties of the Bogoliubov measure $\mu_B$, of functional integrals with respect to $\mu_B$, and of trajectories of the Bogoliubov process were studied in \cite{Sank:2001a,Sank:2001b,Sank:2005,Fatalov:2008}. Using the unit and exponential weights, Pusev \cite{Puse:2010} found the exact small deviation asymptotics for this process, as well as for the $m$-times integrated Bogoliubov process. We shall now state and prove a weighted version of his result.

\begin{theorem}
Consider the Bogoliubov process $Y(t)$, $t\in [0,1],$ and let $\psi$ be a regular-shaped weight on $[0,1]$.
Then as $\varepsilon\to0$
\begin{equation}
\label{psiWeightedBogoliubov}
\Prob\{\|Y\|_\psi\leq\varepsilon\}\sim
\frac{8\sh(\omega/2)\psi^{1/8}(0)\psi^{1/8}(1)}{\vartheta\pi^{1/2}\left(\psi^{1/2}(0)+\psi^{1/2}(1)\right)^{1/2}}
\varepsilon
\exp\left(-\frac{\vartheta^2}{8}\varepsilon^{-2}\right).
\end{equation}
\end{theorem}

\begin{proof}
By Lemma \ref{mainlemma}  we have $\lambda_k=\mu_k^{-1}$, where  $\mu_k$~are  eigenvalues of the boundary-value problem
$$
\left\{
\begin{aligned}
&y''+(\mu\psi-\omega^2)y=0\quad \text{íà} \quad [0,1],\\
&y(0)-y(1)=y'(0)-y'(1)=0.\\
\end{aligned}
\right.
$$

Let $\varphi_{1,2}(t,\zeta)$ be solutions of the equation $y''+(\zeta^2\psi-\omega^2)y=0$ satisfying the initial conditions (\ref{fs1})--(\ref{fs2}).

Substituting the general solution  $y(t)=c_1\varphi_1(t,\zeta)+c_2\varphi_2(t,\zeta)$ into the boundary conditions, we get $\mu_k=x_k^2$, where $x_1<x_2<\ldots$ are  positive zeros of the function
\begin{align*}
F(\zeta)
&=\det
\begin{bmatrix}
\varphi_1(0,\zeta)-\varphi_1(1,\zeta)
&
\varphi_2(0,\zeta)-\varphi_2(1,\zeta)
\\
\varphi'_1(0,\zeta)-\varphi'_1(1,\zeta)
&
\varphi'_2(0,\zeta)-\varphi'_2(1,\zeta)
\end{bmatrix}=
\\
&=1-\varphi_1(1,\zeta)-\varphi'_2(1,\zeta)+\varphi_1(1,\zeta)\varphi'_2(1,\zeta)-\varphi'_1(1,\zeta)\varphi_2(1,\zeta).
\end{align*}

Using Theorem 1.2 of \cite{Naza:2009}, we obtain as in Theorem~\ref{mainthm} that
$$
\Prob\{\|Y\|_\psi\leq\varepsilon\}\sim
C_{\rm dist}
\frac{4}{\sqrt\pi\vartheta}
\varepsilon
\exp\left(-\frac{\vartheta^2}{8}\varepsilon^{-2}\right),
$$
where
$$
C_{\rm dist}^2=
\prod_{k=1}^\infty
\frac{x_k^2}{(\pi (k-1/2)/\vartheta)^2}.
$$

It remains to verify that
$$
C_{\rm dist}^2=
\frac{4\sh^2(\omega/2)\psi^{1/4}(0)\psi^{1/4}(1)}{\psi^{1/2}(0)+\psi^{1/2}(1)}.
$$

As in the proof of Theorem \ref{weightedOU},  we get relations  (\ref{phi1})--(\ref{phi4})   for $\varphi_{1,2}(1,\zeta)$ and $\varphi'_{1,2}(1,\zeta).$
Hence, as $|\zeta|\to\infty$
$$
F(\zeta)=\left(-\left(\frac{\psi^{1/4}(0)}{\psi^{1/4}(1)}+\frac{\psi^{1/4}(1)}{\psi^{1/4}(0)}\right)\cos(\vartheta\zeta)+2\right)(1+O(\zeta^{-1})).
$$

Applying Jensen's theorem to the functions $F(\zeta)$ and $\Psi(\zeta)=\cos(\vartheta\zeta)$
and arguing as in the proof of \cite[Th. 2]{Puse:2010}, we obtain
$$
C_{\rm dist}^2=
\frac{|F(0)|\psi^{1/4}(0)\psi^{1/4}(1)}{\psi^{1/2}(0)+\psi^{1/2}(1)}.
$$
Since the dependence of the solution of differential equation on a parameter is continuous, it follows that
$$
F(0)=\lim_{t\to1}(1-\varphi_1(t,0)-\varphi'_2(t,0)+\varphi_1(t,0)\varphi'_2(t,0)-\varphi'_1(t,0)\varphi_2(t,0)).
$$
It is easy to see that $\varphi_1(t,0)=\ch(\omega t)$, $\varphi_2(t,0)=\frac1\omega\sh(\omega t)$.
Therefore $|F(0)|=4\sh^2(\omega/2)$.
\end{proof}

\begin{rem}
The substitution of  $\psi\equiv1$ and $\psi(t)=\exp(qt)$  into relation (\ref{psiWeightedBogoliubov}) yields the statements of Theorem 1 and Theorem 2 of \cite{Puse:2010}, respectively.
\end{rem}

\section{Small deviations of Bessel processes }

Now we proceed with  the study of small deviations for  functionals  of certain non-Gaussian processes  related to the Brownian motion.
For a non-Gaussian process $X(t)$
the use of random series  of Karhunen-Lo\`{e}ve-type becomes problematic. The problem simplifies if the process $X(t)$ can be expressed by means of simple Gaussian processes.

We start with a Bessel process $Bes^{\delta}$ of dimension $\delta >0$ that corresponds to the  index $\nu = \delta/2 -1 \in (-1, \infty)$. The definition of Bessel processes and  main facts from their theory can be found in  \cite{Boro:Salm:2000}, \cite{Rev:Yor:2001}. For the integer dimension $n$, the process  $Bes^n$ can be viewed as the radial part of the $n$-dimensional Brownian motion with independent components. In the case of Bessel bridge $Bes_0^n$ the components are  independent Brownian bridges.

Small deviation asymptotics for Bessel processes and Bessel bridges are derived from the exact distributions given in \cite{Boro:Salm:2000}. It follows from formula (4.1.9.4 (1)) of \cite{Boro:Salm:2000} that for any dimension $\delta \geq 2$ or any  index $\nu >0$ we have, as  $r \to 0$
\begin{equation}
\label{Bes}
\Prob\{\|Bes^{\delta}\| \leq r\}\sim  \frac{2^{q + 3/2}}{ (q+1)\sqrt{\pi}}\ r \exp \left(- \frac{(q+1)^2}{2r^2}\right) =  \frac{2^{\frac{\delta+3}{2}}}{ \delta\sqrt{\pi}}\ r \exp \left(- \frac{\delta^2}{8r^2}\right).
\end{equation}

  Small deviation asymptotics for Bessel processes and Bessel bridges in  $L_p$-norms were studied in \cite{Fata:2007}.  Formula (\ref{Bes}) follows from the paper \cite{Fata:2007} when $p=2.$

The exact distributions of quadratic norms of Bessel bridges were found by Kiefer \cite{Kie:1959} in connection with some problems of nonparametric statistics. According to Kiefer's formula for any natural $k$ and any  $a > 0$
\begin{equation}
\label{Kiefer}
\Prob\{ \|Bes^{k}_0\|^2 \leq a\} =
\frac{2^{\frac{k+1}{2}}}{\sqrt{\pi}a^{k/4}} \sum_{j=0}^{\infty} \frac{\Gamma(j + \frac{k}{2})}{j!\Gamma(\frac{k}{2})}\exp\left(-\frac{(j + \frac{k}{4})^2}{a}\right) D_{\frac{k-2}{2}} \left( \frac{2j+\frac{k}{2}}{\sqrt{a}}\right),
\end{equation}
where $D_p$ are the functions of parabolic cylinder.

It is clear that for a fixed index $p$ the main role in the asymptotics as $a \to 0$ is played by the first term, so that according to \cite[Sec. 8.4]{Bate:Erd:1974}  we have as $ z \to \infty:$
$$
D_p(z) \sim z^p \exp( -z^2 / 4), \ a \to \infty.
$$
Therefore, for any natural $k$ as $\varepsilon \to 0$
\begin{equation}
\label{Besbridge}
\Prob\{ \|Bes^{k}_0\| \leq \varepsilon \} \sim \frac{2 \sqrt{2}}{\sqrt{\pi}} k^{\frac{k-2}{2}}
\varepsilon^{-(k-1)} \exp\left(- \frac{k^2}{8\varepsilon^2}\right).
\end{equation}

In case of an arbitrary dimension $\delta > 1,$ formulas (\ref{Kiefer}) and (\ref{Besbridge}) remain valid. This can be seen
 by combining  formulas (4.1.0.6) and (4.1.9.8) of the handbook \cite{Boro:Salm:2000}.

\smallskip

Now we shall touch the topic of small deviations of the {\it supremum} of Bessel processes and bridges. Denote by $\mu(Z)$ the supremum of a random process $Z$ on $[0,1].$ The exact formula for the  supremum of the Bessel bridge of integer dimension was found by Gikhman \cite{Gi:1957} and independently by Kiefer \cite{Kie:1959};  later Pitman and Yor \cite{Pitman:Yor:99}  proved its validity for any positive dimension $\delta > 0$ or any index  $\nu = \frac{\delta -2}{2}.$

 Let $0<j_{\nu,1} < j_{\nu,2} <... $ be the sequence of positive zeros of the Bessel function $J_{\nu}.$  The Gikhman-Kiefer-Pitman-Yor formula says that for $\nu >-1$ (that is for $\delta > 0$) and any $r\geq 0$ we have
$$
\Prob\{ \mu(Bes_0^\delta) \leq r\} = \left( 2^{\nu -1} \Gamma(\nu +1) r^{\nu} \right)^{-1} \sum_{n=1}^{\infty} \frac{ j_{\nu, n}^{2\nu} }{ J_{\nu+1}^2(  j_{\nu, n} ) }
\exp\left( - \frac{j_{\nu,n}^2}{2r^2} \right).
$$
Obviously the main contribution to the series as $r \to 0$ is made by the first term. Therefore the exact small deviation asymptotics for the supremum of the Bessel bridge
 with  $\nu > -1$ looks as follows: 
\begin{equation}
\label{J1}
\Prob\{ \mu(Bes_0^\delta) \leq  r \}  \sim  \left( 2^{\nu -1} \Gamma(\nu +1) r^{2\nu+2} \right)^{-1} \frac{ j_{\nu, 1}^{2\nu} }{ J_{\nu+1}^2 (  j_{\nu, 1} ) }
\exp\left( - \frac{j_{\nu,1}^2}{2r^2} \right).
\end{equation}

An analogous formula for the Bessel process has a slightly different  form. It follows from the handbook \cite[formula 4.1.1.4]{Boro:Salm:2000}, see also \cite{Pitman:Yor:99},  that  for  $\nu > -1$
$$
\Prob\{ \mu(Bes^\delta) \leq  r \} = \left( 2^{\nu -1} \Gamma(\nu +1) \right)^{-1} \sum_{n=1}^{\infty} \frac{ j_{\nu, n}^{\nu -1} }{ J_{\nu+1}(  j_{\nu, n} ) }
\exp\left( - \frac{j_{\nu,n}^2}{2r^2} \right).
$$
This formula for integer dimensions was firstly obtained in the famous paper by Ciesielski and Taylor \cite{Ciesielski:62}. We again observe that the main contribution as $r \to 0$ is made by the first term. Consequently, we get the exact small deviation asymptotics as $r\to 0$:
\begin{equation}
\label{J2}
\Prob\{ \mu(Bes^\delta) \leq  r \}  \sim \left( 2^{\nu -1} \Gamma(\nu +1) \right)^{-1}  \frac{ j_{\nu, 1}^{\nu -1} }{ J_{\nu+1}(  j_{\nu, 1} ) }
\exp\left( - \frac{j_{\nu,1}^2}{2r^2} \right).
\end{equation}

Special cases of  formulas  (\ref{J1}) and  (\ref{J2}) for small integer dimensions are discussed in the survey of Fatalov \cite{Fata:2003}.

It is well-known that the powers of Bessel processes belong to the same family of processes up to a suitable time change  \cite[Ch. XI]{Rev:Yor:2001}. This leads to numerous identities in law between the integrals of different powers of Bessel processes, see \cite[Ch. XI]{Rev:Yor:2001}, \cite{Biane:Yor:87}, \cite{Biane:Yor:2001}. We select, as an example, a typical identity \cite[Coroll. 1.12, Ch. XI]{Rev:Yor:2001}, which is valid for integer dimensions $d>1$: 
$$
4 \left( \int_0^1 \left( Bes^d (s) \right) ^{-1} ds\right)^{-2}  \stackrel{law} =    \int_0^1 \left(Bes^{2d-2} (s)\right)^2 ds   = ||Bes^{2d-2}||^{2}.
$$
Due to this identity the exact  small deviation asymptotics (\ref{Bes}), which is valid for the right-hand side, is also  valid for the left-hand side.
One more example is given by the identity in law from \cite[Table 2]{Biane:Yor:2001}:
\begin{equation}
\label{Bessel4}
|| Bes_0^4||^2 = \int_0^1 (Bes_0^4 (s) )^2 ds  \stackrel{law} =  \left(    \frac{1}{\pi} \int_0^1 (Bes_0^3 (s) )^{-1} ds  \right)^2.
\end{equation}
Using  formula  (\ref{Besbridge}), we get a new asymptotic relation as $r \to 0$:
\begin{equation}
\label{Bes2}
\Prob\left\{ \int_0^1 (Bes_0^3 (s) )^{-1} ds \leq  r \right\} \sim \frac{8\pi^{2} \sqrt{2\pi}}{r^{3} }\exp\left(- \frac{2\pi^2}{r^2}\right).
\end{equation}

\medskip

Applying  Lemma \ref{LemmaLifs3} to the squares of weighted norms of Brownian motions and using Theorem \ref{weightedW}, we obtain the exact small deviation asymptotics for the weighted norm of a Bessel process of integer dimension.

\begin{proposition}
\label{Beskpsi}
Let $\psi$ be a regular-shaped weight on $[0,1]$. Then as $\varepsilon\to0$
$$
\Prob\{\|Bes^k\|_\psi\leq\varepsilon\}\sim
\frac{2^{(k+3)/2}\psi^{k/8}(0)}{\sqrt\pi k\vartheta\psi^{k/8}(1)}
\varepsilon
\exp\left(-\frac{k^2\vartheta^2}{8}\varepsilon^{-2}\right).
$$
\end{proposition}

In a similar way, we can obtain the exact small deviation asymptotics for the weighted $L_2$-norm of a Bessel bridge of integer dimension.
\begin{proposition}
\label{Bes0kpsi}
Let $\psi$ be a regular-shaped weight on $[0,1]$.  Then as $\varepsilon\to0$
$$
\Prob\{\|Bes_0^k\|_\psi\leq\varepsilon\}\sim
\frac{2\sqrt2}{\sqrt\pi}
(k\vartheta)^{(k-2)/2}\psi^{k/8}(0)\psi^{k/8}(1)
\varepsilon^{-(k-1)}
\exp\left(-\frac{k^2\vartheta^2}{8}\varepsilon^{-2}\right).
$$
\end{proposition}

For use later on, we describe the small deviation asymptotics of $\|Bes^k\|^2 + \|Bes_0^m\|^2 $ for integer $k,m\geq 1$.
\begin{lemma}
\label{BeskBes0m}
Let $Bes^k$ be a Bessel process of integer dimension $k \geq 1$ and let  $Bes_0^m$ be an independent of $Bes^k$ Bessel bridge of integer dimension $m \geq 1$. Then as $\varepsilon \to 0$
$$
\Prob\{ \|Bes^{k}\|^2 + \|Bes_0^{m}\|^2 \leq \varepsilon^2 \} \sim  \frac{2^{\frac{k +3}{2}} m^{m-1}}{\sqrt{\pi}(k+m)^{m/2}} \varepsilon^{m-1} \exp( - \frac{(m+k)^2}{8\varepsilon^2}).
$$
\end{lemma}

\begin{proof}
The proof consists of ``pasting together''  two asymptotics, (\ref{Bes}) and (\ref{Besbridge}), with the help of Lemma \ref{LifshitsLemma}. After some calculations we get the statement of the lemma.
\end{proof}

Another example of the use of Lemma \ref{LemmaLifs3} is as follows. Consider a non-summable weight  $\omega(t) = [t(1-t)]^{-1}, 0 \leq t \leq 1$, called the Anderson-Darling weight, see  \cite{Anderson:1952}. The squared norm of the Brownian bridge with the Anderson-Darling weight (in the notation of \cite{Mansuy:2005})
$$
A^{1,1}  := \int_0^1 B^2(t) \  \omega(t) \ dt
$$
plays an important role in the theory of distribution-free goodness-of-fit tests. It was proved in \cite{Naza:2003}  that as $r \to 0$
\begin{equation}
\label{Nazar}
\Prob\{ A^{1,1} \leq r \} \sim \frac{2}{\sqrt{r}} \exp( - \frac{\pi^2}{8r}).
\end{equation}
Now, following the paper \cite{Mansuy:2005}, for natural $n$ consider  the functional $$A^{n,1} := \int_0^1  ( Bes_0^n(t))^2 /  t(1-t) \ dt.    $$ Clearly $A^{n,1}$ is the sum of  $n$ independent copies of  $A^{1,1}.$ Hence by Lemma \ref{LemmaLifs3} and (\ref{Nazar}) we get as $r \to 0$
\begin{equation}
\label{Mansuy}
\Prob\{ A^{n,1} \leq r \} \sim 2^{\frac{3-n}{2}}  (n\pi\sqrt{\pi})^{n-1} \ r^{-n +\frac12} \exp\left( - \frac{n^2\pi^2}{8r}\right).
\end{equation}

Now  denote by $T_1 (3) $ and  $T_1 (4) $ the  first  hitting times  of the level 1 by the Bessel processes $Bes^3$ and $Bes^4$,
 respectively. The following two identities in law are proved in \cite[p.~174 ]{Mansuy:2005}:
$$
A^{2,1}  \stackrel{law}{=} \int_0^{T_1(3)}  \frac{ds} { Bes^3(s) (1-Bes^3 (s)) }  \stackrel{law}{=} 4 \int_0^{T_1(4) } \frac{du}{1 - (Bes^4(u))^2} .
$$

The use of relation (\ref{Mansuy}) with $n=2$ gives us two new exact  asymptotics for small deviation probabilities of Bessel processes  as $r\to 0$:
\medskip
$$
\begin{array}{ll}
\medskip
\Prob\{ \int_0^{T_1(3)}  \left( Bes^3(s) (1-Bes^3 (s))\right)^{-1} ds \leq r \} \sim   2\pi\sqrt{2\pi} \ r^{- \frac32} \exp\left( - \pi^2/ 2r\right), \\
\Prob\{ \int_0^{T_1(4) }  \left( 1 - (Bes^4(u))^2 \right) ^{-1} \ du  \leq r \} \sim  \frac14\pi\sqrt{2\pi} \ r^{- \frac32} \exp\left( - \pi^2 / 8r\right).
\end{array}
$$

In the conclusion of this section, consider the remarkable result of Alili  \cite{Alili:97}, \cite[formula 4.33]{Biane:Yor:2001},
which establishes the following identity in law for any $\sigma \neq 0:$
\begin{equation}
\label{Alili}
\frac{\sigma^2}{\pi^2} \left[ \left( \int_0^1 \coth ( \sigma Bes^3_0(u) ) du \right)^2 - 1\right] \stackrel{law}= ||Bes_0^4||^2.
\end{equation}
Using (\ref{Bessel4}) this can be equivalently written  in the form
$$
\sigma^2 \left[ \left( \int_0^1 \coth ( \sigma Bes^3_0(u) ) du \right)^2 - 1\right] \stackrel{law}=   \left(  \int_0^1 (Bes_0^3 (s) )^{-1} ds  \right)^2.
$$
  Using (\ref{Besbridge}) we can obtain the exact small deviation asymptotics for the functional on the left-hand side of (\ref{Alili}) for any $\sigma.$  As shown in \cite{Alili:Yor:97},  see also \cite[formula 4.34]{Biane:Yor:2001}, the following surprising identity in law holds as $\sigma \to \infty:$
$$
\int_0^{\infty} \left(\exp(Bes^3(t)) -1\right)^{-1}  dt  \stackrel{law} = \  \frac{\pi^2}{4} ||Bes_0^2||^2.
$$

This result implies one more exact small deviation asymptotics, now for the exponential functional of Bessel process.
\begin{proposition}
As $r \to 0$ we have
$$
\Prob\left\{   \int_0^{\infty} \left(\exp\left(Bes^3(t)\right) -1\right)^{-1}  dt  \leq r \right\} \sim   \sqrt{ \frac{2\pi}{r} } \exp \left( - \frac{\pi^2}{8r} \right).
$$
\end{proposition}

\section{Small deviations for the Brownian excursion}

This section addresses the topic of small deviations of  Brownian excursion. Denote by $\mathfrak e(t)$, $0\leq t\leq1$, a usual Brownian excursion. The accurate definition of this process can be found, e.g., in \cite{Boro:Salm:2000}, \cite{Roge:Will:2000}, or \cite{Rev:Yor:2001}.   Informally, we may think of the normalized Brownian excursion on the interval [0,1] as the Brownian bridge from zero to zero taking positive values within this interval or as the Brownian motion starting at zero, conditioned to stay positive  and to hit zero for the first time at time 1.

 For our purposes the link of the Brownian excursion with the Bessel bridges plays  the key role. This link is known long ago and was partially described  by L\'{e}vy \cite{Levy:1965}, by It\^{o} and McKean  \cite[Sec. 2.9]{Ito:McKean:1965}, and by Williams \cite{Will:1970}. Later, the Brownian excursion and related processes were studied in numerous papers, see, e.g., \cite{Rev:Yor:2001} and \cite{Man:Yor:2008}. However, to the best of our knowledge,  their exact small deviation asymptotics in $L_2$-norm have not yet been obtained.

\begin{theorem}
\label{excursion}
Let $\psi$ be a regular-shaped weight on $[0,1]$. Then as $\varepsilon\to0$
$$
\Prob\{\|\mathfrak e\|_\psi\leq\varepsilon\}\sim
\frac{2\sqrt{6\vartheta}\psi^{3/8}(0)\psi^{3/8}(1)}{\sqrt\pi}
\varepsilon^{-2}
\exp\left(-\frac{9\vartheta^2}{8}\varepsilon^{-2}\right).
$$
\end{theorem}

\smallskip

\begin{proof}
The following identity in law, sometimes called the L\'{e}vy-Williams identity \cite[p.~454]{Biane:Yor:2001}, is well known \cite{Will:1970},  \cite{Roge:Will:2000}:
\begin{equation}
\label{Levy}
\{\mathfrak e^2(t),0\leq t\leq1\}  \stackrel{law}{=} \{B_1^2(t)+B_2^2(t)+B_3^2(t),0\leq t\leq1\},
\end{equation}
where $B_1(t)$, $B_2(t)$, $B_3(t)$, $0\leq t\leq1$, are three independent Brownian bridges.
Multiplying both parts of (\ref{Levy}) by $\psi$ and integrating, we get
$$
\|\mathfrak e\|^2_\psi\stackrel{law}{=}\|B_1\|^2_\psi+\|B_2\|^2_\psi+\|B_3\|^2_\psi \stackrel{law}=\|Bes_0^3\|^2_\psi.
$$
The application of Proposition \ref{Bes0kpsi} with $k=3$ completes the proof. 
\end{proof}

Choosing the unit weight in Theorem \ref{excursion} leads to the exact small deviation asymptotics for the Brownian excursion itself:
\begin{equation}
\label{excursionitself}
\Prob\{\|\mathfrak e\| \leq\varepsilon\}\sim
\frac{2\sqrt{6}}{\sqrt\pi}
\varepsilon^{-2}
\exp\left(-\frac{9}{8}\varepsilon^{-2}\right), \quad \varepsilon \to 0.
\end{equation}

As a one more example, consider the smooth weight  $\chi(t) = (1+t)^{-4}$. Applying Theorem \ref{weightedBridge}, we easily obtain (here $\vartheta = \frac12$) that as $\varepsilon\to0$
$$
\Prob\{\|B\|_\chi\leq\varepsilon\}\sim
\frac{2\sqrt{2}}{\sqrt{\pi}}
\exp\left(-\frac{1}{32}\varepsilon^{-2}\right)
$$
(the result also follows from Theorem 4 of \cite{Naza:Puse:2009} with $a=1$).

Incidentally, this relation corrects the erroneous formula (13) of \cite{Gutierrez:1987}.
Considering the excursion with the weight $\chi(t)$, we get from Theorem \ref{excursion} that
$$
\Prob\{\|\mathfrak e\|_\chi\leq\varepsilon\}\sim
\sqrt{\frac{3}{2\pi}}
\varepsilon^{-2}
\exp\left(-\frac{9}{32}\varepsilon^{-2}\right).
$$

As in the proof of  Theorem \ref{excursion}, using Theorems 3.1 and 3.3 of \cite{Naza:2003}, and Example c) from \cite{Niki:Khar:2004},  we
can establish the exact small deviation asymptotics for the Brownian excursion with various ``degenerate'' weights.

\begin{proposition}
The following assertions hold as $\varepsilon\to0:$
$$
\Prob\left\{\int_0^1\frac{\mathfrak e^2(t)}{t(1-t)}\,dt
\leq\varepsilon^2\right\}\sim
\frac{9\pi^3}{\varepsilon^5}
\exp\left(-\frac{9\pi^2}{8\varepsilon^2}\right),
$$
$$
\Prob\left\{\int_0^1\frac{\mathfrak e^2(t)}{t(2-t)}\,dt
\leq\varepsilon^2\right\}\sim
\frac{3^{7/8}\pi^{3/2}}{2^{5/4}}
\varepsilon^{-11/4}
\exp\left(-\frac{9\pi^2}{32\varepsilon^2}\right).
$$
Let $\theta>-2$. Then as $\varepsilon\to0$
\begin{multline*}
\Prob\left\{\int_0^1t^\theta \mathfrak e^2(t)\,dt
\leq\varepsilon^2\right\}\sim
\frac{4\pi^{1/4}}{3^{(\theta-4)/(4(\theta+2))}\Gamma^{3/2}\left(\frac{\theta+3}{\theta+2}\right)}
\times\\\times
\left((\theta+2)\varepsilon\right)^{-\frac{\theta+8}{2(\theta+2)}}
\exp\left(-\frac92\left((\theta+2)\varepsilon\right)^{-2}\right).
\end{multline*}
\end{proposition}

The following identity in law for the Watson-type  functional of the Brownian excursion holds, see \cite[formula (4.20)]{Biane:Yor:2001}:
$$
 Wat ({\mathfrak e}) := \int_0^1 \left( {\mathfrak e}(t) - \int_0^1 {\mathfrak e}(x) dx\right)^2 dt \stackrel{law}=  \frac14  ||Bes_0^2||^2 .
$$
Applying formula (\ref{Besbridge})  with $k=2,$ we obtain a new exact asymptotics as $r \to 0$:
$$
\Prob\{ Wat ({\mathfrak e}) \leq r \} \sim  \sqrt{\frac{2}{\pi r}}
 \exp\left(- \frac{1}{8r}\right).
$$

\section {On small deviations of the Brownian local time}

Now, let $L_t^x (B)$ be the jointly continuous local time of a Brownian bridge $B$ at the point $x\in \mathbb R$ up to time $ t\in [0,1].$
By virtue of Corollary 2.2 of \cite{Cs:Shi:Yor:1999}, one has the equality in law for any natural  $m$
\begin{equation}
\label{loctime}
\int_{-\infty}^{\infty} (L_1^x(B))^m dx  \stackrel{law}{=} 2^{m-1}  \int_0^1 ( \mathfrak e(t))^{m-1} dt.
\end{equation}

First, consider the case $m=3$.
In view of  Theorem \ref{excursion}, the small deviation asymptotics of the functional $\|\mathfrak e\|^2$ on the right-hand  side of
(\ref{loctime}) are known. This implies the exact small deviation asymptotics for the integral functional $\int_{-\infty}^{\infty} (L_1^x(B))^3 dx$.
\begin{proposition}
The following relation holds as $\varepsilon \to 0:$
$$
\Prob\left\{\int_{-\infty}^{\infty} (L_1^x(B))^3 dx \leq\varepsilon\right\}\sim
\frac{8\sqrt6}{\sqrt\pi}
\varepsilon^{-1}
\exp\left(-\frac92\varepsilon^{-1}\right).
$$
\end{proposition}
\noindent This proposition refines on the result of \cite{Cs:Shi:Yor:1999}, where
the asymptotic relation was proved at the logarithmic level only.

Now we proceed to the case $m=2$. We have the following equality in law
 $$
\int_{-\infty}^{\infty} (L_1^x(B))^2 \ dx  \stackrel{law}{=} 2  \int_0^1  \mathfrak e(t) dt.
$$
The integral on the right-hand side, which may be interpreted as the  Brownian excursion area, has been studied by many authors, see a history of the question in \cite{Jan:2007}. The distribution  of this integral can be described as follows.

Let  $Ai(x)$ be a standard Airy function \cite[Ch. 10]{Abr:1979}. All zeros of the Airy function are negative. Denote them by $-a_j$, $j \geq 1$, and let $a_1$ be the absolute value of the first zero, which is approximately $ 2.3381$. The following asymptotic expansion is valid  as $r\to 0+$
\cite[Sec. 15]{Jan:2007}:
$$
\Prob\left\{\int_0^1  \mathfrak e(t) dt \leq r \right\}
\sim \exp\left(-\frac{2a_1^3}{27r^2}\right)\left( \frac23 a_1^{3/2}r^{-2} + \frac14 a_1^{-3/2} - \frac{105}{64} a_1^{-9/2}  r^2 +... \right),
$$
and the main contribution is made by the first term.  Hence the following exact small deviation asymptotics for the functional $\int_{-\infty}^{\infty} (L_1^x(B))^2 \ dx $ holds.

\begin{proposition}
As $\varepsilon \to 0$
$$
\Prob\left\{\int_{-\infty}^{\infty} (L_1^x(B))^2 \ dx \leq \varepsilon \right\} \sim
\frac83 a_1^{3/2}\varepsilon^{-2} \exp\left(-\frac{8 a_1^3}{27\varepsilon^2}\right).
$$
\end{proposition}

A less accurate version of this result, at the logarithmic level only, was obtained in   \cite[Theorem 3.1]{Cs:Shi:Yor:1999}.
The exact asymptotics given here seems to appear in the literature for the first time. In general,
the results on small deviations of Brownian local times are sparse. We augment them by the corollary of the first Ray-Knight theorem, see, e.g., \cite[Ch. XI]{Rev:Yor:2001}, \cite[Ch. 3]{Man:Yor:2008}.

Let  $T_1 = \inf \{ t: W(t) =1\}$ be the first  hitting time of 1 by the Brownian motion. For $x \in [0,1]$  consider the local time process in $x$ up to the moment $T_1 $:
$$Z(x) = L_{T_1}^x (W), \quad  0 \leq x \leq 1.$$
Ray-Knight's first theorem says that on the interval $[0,1]$ this process equals in law to the  square of the process $Bes^2$. Therefore we have
$$
\Prob\left\{\int_{0}^{1}   L_{T_1}^x(W) \ dx  \leq \varepsilon^2 \right\}  \sim   \frac{2\sqrt{2}}{ \sqrt{\pi}}\ \varepsilon \exp \left(- \frac{1}{2\varepsilon^2}\right).
$$
The generalization of this result via Proposition \ref{Beskpsi} for the weighted quadratic norm  looks as follows.
\begin{proposition}
Let $\psi$ be a regular-shaped weight on $[0,1]$. Then as $\varepsilon\to0$
$$
\Prob\left\{\int_{0}^{1}   L_{T_1}^x(W) \psi(x) \ dx  \leq \varepsilon^2 \right\}\sim
\frac{2\sqrt2\psi^{1/4}(0)}{\sqrt\pi\vartheta\psi^{1/4}(1)}
\varepsilon
\exp\left(-\frac{\vartheta^2}{2}\varepsilon^{-2}\right).
$$
\end{proposition}

Using the relationship between the process $L_{T_1}^x(W)$ and the two-dimensional Bessel process, we can find the exact small deviation asymptotics  of the  $L_p$-norm of the process  $L_{T_1}^x(W)$. For any positive $p$, we have
$$
\Prob\left\{\int_{0}^{1}   (L_{T_1}^x(W))^p \ dx  \leq \varepsilon^p \right\}=
\Prob\left\{\int_{0}^{1}   (Bes^2(t))^{2p} \ dt  \leq \varepsilon^{p} \right\},
$$
where the asymptotics for the right-hand side, in implicit form, are obtained from  \cite{Fata:2007}.

\medskip

An interesting process related to a random process  $X(t)$ starting from zero is the so-called supremum-process $S(t) = \sup_{0 \leq s \leq t} X(s).$
In case of the Brownian motion, the classical result by L\'{e}vy  \cite{Levy:1965} says that the process  $ S(t) - W(t)$, $t\geq 0$,
coincides in law with a reflected Brownian motion $|W|$, that is, with a one-dimensional Bessel process.
Then, the application of Theorem \ref{weightedW} leads to the  following proposition.
\begin{proposition}
\label{Levy2}
Let $\psi$ be a  regular-shaped weight on $[0,1]$. Then  as $\varepsilon\to0$
$$
\Prob\{\|S - W\|_\psi\leq\varepsilon\}\sim
\frac{4\psi^{1/8}(0)}{\sqrt\pi\vartheta\psi^{1/8}(1)} \varepsilon
\exp\left(-\frac{\vartheta^2}{8}\varepsilon^{-2}\right).
$$
\end{proposition}

A more remarkable result belongs to Pitman  \cite{Pitm:1975},  see also \cite[Ch. 6]{Rev:Yor:2001}, who proved the coincidence in law of the process $2S - W$ and the three-dimensional Bessel process $Bes^3.$ Pitman's identity implies the relation
$$
\Prob\{\|2S - W\|_\psi\leq\varepsilon\} = \Prob\{\|Bes^3\|_\psi\leq\varepsilon\}.
$$
Hence, we can obtain  via Lemma \ref{LemmaLifs3} the exact small deviation asymptotics for the Pitman process with respect to weighted norm.

Another interesting fact from \cite[Ch.6, Coroll. 3.8]{Rev:Yor:2001} consists of the equality in law of the processes $\{|W|(t) + L_t^0(W), t\geq 0\}$ and $\{Bes^{3} (t)$,$ t\geq 0\}$. The use of this fact yields the exact asymptotics as $\varepsilon \to 0$ for the probability
$$
\Prob\{\| \ |W|(t) + L_t^0(W) \|_\psi\leq\varepsilon\}.
$$

As to the supremum in $x$ of the Brownian local time $L_t^x(W)$,  its distribution is well-studied and can be found in
\cite[formula 1.11.4 ]{Boro:Salm:2000}. Therefore we have as $\varepsilon  \to 0$
$$
\Prob\{ \sup_{x \in {\mathbb{R}} } L_t^x (W) \leq \varepsilon  \}  \sim  \frac{4}{\sin^2 (j_{0,1})} e^{ - 2j_{0,1}^2 t / \varepsilon^2}.
$$

It is also of interest to look at the supremum of Brownian local time up to some hitting time. Let $\tau_b = \inf\{s: W(s) =b\} $ be the first hitting time of the level $b \in {\mathbb{R}}.$ \ Then, as proved in  \cite[Ch. I,  formula (4.13)]{Boro:89},  see also \cite[formula (2.11.2)]{Boro:Salm:2000}, the following exact asymptotics holds as
$\varepsilon \to 0:$
$$
\Prob\{ \sup_{-\infty < x <b } L_{\tau_b}^x (W)     \leq \varepsilon  \}  \sim  \frac{8}{ j_{0,1}^3 J_1 ( j_{0,1})} e^{ - bj_{0,1}^2/\varepsilon^2}.
$$

\section { On small deviations of Brownian meander}

In this section, we consider the Brownian meander ${\mathfrak m}(t).$ A rigorous definition of the meander of length 1 can be found in  \cite[Ch. XII]{Rev:Yor:2001} and \cite[p.~83]{Boro:Salm:2000}.   The Brownian meander can be thought of as a Brownian motion conditioned to
stay positive up to time 1, but it is not required that  the Brownian meander vanishes at the point 1. Denote by $\mathfrak m^z (t)$  the Brownian meander taking the value  $z \geq 0$ at the point 1.

The following fact is proved in  \cite{Ber:Pit:Ru:1999}:
$$
({\mathfrak m}^z (t))^2 \stackrel{law}{=} B_1^2(t) +B_2^2(t) + (B_3(t) +zt)^2, \, 0\leq t \leq 1,
$$
where $B_1$, $B_2$, and $B_3$ are three  independent Brownian bridges. For  $z=0$ we get, as a particular case, the L\'{e}vy-Williams identity discussed above.

First, we need the exact small deviation asymptotics of $(B(t) +zt)^2$. For this, we can use, e.g., the result from \cite{Mi:1994} saying that  as $r \to 0$
\begin{equation}
\label{zzz}
\Prob\left\{\int_{0}^{1} (B(t) +zt)^2 \ dt \leq \ r \right\} \sim
\sqrt{\frac{8 }{\pi (1+z^2) } }\ \exp\left(- \frac{ (z^2 +1)^2 }{8r } + \frac{z^2}{2} \right).
\end{equation}
Due to  (\ref{Besbridge}),  we have as $r \to 0:$
$$
\Prob\left\{\|B_1\|^2 + \|B_2\|^2 \leq r \right\} \sim \frac{2\sqrt{2}}{\sqrt{\pi r}}\exp\left(-\frac{1}{2r}\right).
$$

Combining this result with (\ref{zzz}) via Lemma  \ref{LifshitsLemma} we get as $\varepsilon \to 0$
\begin{multline*}
\Prob\left\{\int_0^1\left( B_1^2(t) +B_2^2(t) + (B_3(t) +zt)^2\right) dt \leq \varepsilon^2 \right\} \sim \\
\sim \frac{2\sqrt{2(z^2+3)}}{\sqrt{\pi}} \ e^{\frac{z^2}{2}}\varepsilon^{-2} \
\exp\left(- \frac{(z^2+3)^2}{8}\varepsilon^{-2}\right).
\end{multline*}

Thus, the following exact small deviation asymptotics for the Brownian meander ${\mathfrak m}^z$ with prescribed end  holds.

\begin{proposition}
\label{meander-end}
For  any $z \geq 0$, we have as $\varepsilon \to 0$
$$
\Prob\{\| {\mathfrak m}^z \| \leq \varepsilon \}  \sim \frac{2\sqrt{2(z^2+3)}}{\sqrt{\pi}} \  \varepsilon^{-2} \
\exp\left(- \frac{(z^2+3)^2}{8}\varepsilon^{-2} + \frac{z^2}{2} \right).
$$
\end{proposition}

For $z=0$ the above result is in perfect accordance with  formula (\ref{excursionitself}) for the Brownian excursion.

We next turn to deriving the small deviations of the usual Brownian meander $\mathfrak{m}(t)$.
To this end, recall the following equality in law, see
 \cite[Corr. 3.9.1]{Man:Yor:2008} or \cite[Sec. 5]{Pitman:Yor:96}:
\begin{equation}
\label{identity}
\{\mathfrak m^2(t),0\leq t\leq1\}  \stackrel{law}{=} \{B^2(t) + W_1^2(t) + W_2^2(t),0\leq t\leq1\},
\end{equation}
where the Brownian bridge $B$ and two Brownian motions $W_1$ and $W_2$ are  independent. The process on the right-hand side of  (\ref{identity}),
is the ``interpolation'' between the squares of the Bessel process and of the Bessel bridge.

Using (\ref{identity}) we readily get the exact small deviation asymptotics for the meander $\mathfrak{m} (t).$
In a particular case of $k=2$ and $m=1$, by Lemma \ref{BeskBes0m} we get the following result.

\begin{proposition}
As $\varepsilon \to 0$ it is true that
$$
\Prob\{\| {\mathfrak m}\|  \leq \varepsilon \} \sim 4 \sqrt{\frac{2 }{3\pi} }\ \exp\left(- \frac98 \varepsilon^{-2}\right).
$$
\end{proposition}

Similarly, we can infer the exact small deviation asymptotics  for the Brownian meander under various weights.

\begin{theorem}
\label{meanweight}
Let $\theta>-2$ be a given number. As $\varepsilon\to0$
\begin{multline*}
\Prob\left\{ \int_0^1 t^\theta \mathfrak m^2(t) dt \leq \varepsilon^2 \right\}\sim
\frac
{2^{2+\frac{\theta}{2(\theta+2)}}\pi^\frac12}
{3^{\frac12+\frac{3\theta}{4(\theta+2)}}\Gamma\left(\frac{1}{\theta+2}\right)\Gamma^{1/2}\left(\frac{\theta+3}{\theta+2}\right)}
\times\\\times
((\theta+2)\varepsilon)^{\frac{3\theta}{2(\theta+2)}}
\exp\left(-\frac94((\theta+2)\varepsilon)^{-2}\right).
\end{multline*}
For a regular-shaped weight function $\psi$ on $[0,1]$, we have
$$
\Prob\{ \|\mathfrak m\|_\psi \leq \varepsilon \} \sim
4\sqrt{\frac{2}{3\pi}}\frac{\psi^{3/8}(0)}{\vartheta^{1/2}\psi^{1/8}(1)}
\exp\left(-\frac{9\vartheta^2}{8}\varepsilon^{-2}\right),
$$
where
$$
\vartheta=\int_0^1\sqrt{\psi(t)}dt.
$$
\end{theorem}

There is an interesting identity in law that relates the Brownian meander  $\mathfrak m(s)$ with  the Brownian bridge $B(s)$ and its local time at zero
$L_{s}^0(B)$. It looks as follows, see \cite[Sec. 8.3]{Man:Yor:2008}:
$$
\{ \mathfrak m (s), s \leq 1 \} \stackrel{law}{=}  \{ |B(s)| + L_{s}^0(B), s \leq 1  \}.
$$
  The use of Theorem \ref{meanweight} in combination with the above identity gives the corresponding asymptotics for the norm $ || \  |B(s)| + L_{s}^0(B)||_{ \psi}.$

\medskip

Now let us find the exact small deviation asymptotics for the supremum of the Brownian excursion $ \mu({\mathfrak e}) = \sup_{0 \leq u\leq 1} {\mathfrak e}(u).$ To do this, denote by $\mathcal T_4$ the square of the norm of the four-dimensional Bessel bridge, that is,
$$
\mathcal T_4 = \int_0^1 \left ( \sum_{i=1}^4 B_i ^2 (t)\right ) dt,
$$
with $B_i$, $i=1,\dots,4,$ being  independent Brownian bridges.
By virtue of Lemma \ref{LemmaLifs3},  we have
$$
\Prob\{\mathcal T_4 \leq r \} \sim \frac{8\sqrt{2}}{\sqrt{\pi r^3}}\exp\left(-\frac{2}{r}\right).
$$
Next, from the fact, see \cite[Sec. 11.3]{Man:Yor:2008},
$$
\mu^2({\mathfrak e}) \stackrel{law}{=} \frac{\pi^2}{4} \ \mathcal T_4,
$$
we deduce a simple but new relation:
\begin{proposition}
\label{supexc}
As $\varepsilon \to 0$ we have
$$
\Prob\{\mu({\mathfrak e}) \leq \varepsilon \} \sim
\frac{\pi^2\sqrt{2\pi}}{\varepsilon^3} \exp\left(- \frac{\pi^2}{2 \varepsilon^{2}}\right).
$$
\end{proposition}

Alternatively, this asymptotics can  be obtained via Lemma \ref{LifshitsLemma} from the well-known Chung identity, see \cite[Sec. 11.3]{Man:Yor:2008}:
$$
\mu^2({\mathfrak e}) \stackrel{law} = \mu^2 (|B_1|) +\mu^2 (|B_2|),
$$
where the independent random variables $\mu (|B_1|)$ and $\mu (|B_2|)$ both have the Kolmogorov distribution, see  \cite[Ch. 2, Sec. 2]{Sh:We:1986}. That is, that for all $r>0$
$$
\Prob\{\mu(|B|) \leq r \} =  \frac{\sqrt{2\pi}}{r} \sum_{k=1}^{\infty} \exp(- \frac{(2k-1)^2\pi^2}{8r^2}).
$$
Indeed, as $r \to 0$
\begin{equation}
\label{Bridge}
\Prob\{ \mu(|B|)\leq r \} \sim  \frac{\sqrt{2\pi}}{r} \exp \left(- \frac{\pi^2}{8r^2}\right),
\end{equation}
and it remains to apply Lemma \ref{LifshitsLemma}. \hfill$\square$\

Proposition \ref{supexc}  implies an interesting relation connected to the integral functional
$$
h({\mathfrak e}) = \int_0^1 ds / {\mathfrak e} (s).
$$

\begin{proposition}
As $\varepsilon \to 0$
$$
\Prob\{ h({\mathfrak e}) \leq \varepsilon \} \sim
\frac{8\pi^2\sqrt{2\pi}}{\varepsilon^3} \exp\left(-\frac{2\pi^2}{\varepsilon^2}\right).
$$
\end{proposition}

The proof follows immediately from the identity in law, see \cite[Sec. 11.5]{Man:Yor:2008},
\begin{equation}
\label{meme}
 h({\mathfrak e}) \stackrel{law} = 2 \mu({\mathfrak e}).
\end{equation}

By using (\ref{Bridge}), one more asymptotic formula can be obtained. Applying the fact, see \cite[p.250]{Pitman:Yor:96},
\begin{equation}
\label{PY}
\int_0^1\left( Bes_0^2 (s) \right) ^{-1} ds   \stackrel{law}= 2 \mu ( |B|),
\end{equation}
 we get  a formula resembling to  (\ref{Bes2}): as $r \to 0$
$$
\Prob\left\{ \int_0^1\left( Bes_0^2 (s) \right) ^{-1} ds \leq r \right\} \sim  \frac{2\sqrt{2\pi}}{r} \exp \left(- \frac{\pi^2}{2r^2}\right).
$$

It is interesting to compare the small deviation asymptotics for the suprema of Brownian excursion and Brownian meander.
Recall the identity, see \cite{Kennedy:76}, \cite[p.449] {Biane:Yor:2001},
$$
\mu( {\mathfrak m} ) \stackrel{law} = 2\mu( |B|),
$$
where  the right-hand side coincides with that of (\ref{PY}).  Then, the application of (\ref{Bridge}) yields:
\begin{proposition}
\label{supmean}
As $\varepsilon \to 0$ we have
$$
\Prob\{\mu({\mathfrak m}) \leq \varepsilon \} \sim  \frac{2\sqrt{2\pi}}{\varepsilon} \exp \left(- \frac{\pi^2}{2\varepsilon^2}\right).
$$
\end{proposition}
The comparison of Propositions  \ref{supexc} and \ref{supmean} shows the equality of exponents and the difference of  power terms and constants,
as expected.

Interestingly, for the Brownian meander equality (\ref{meme})  takes a similar but different form \cite{Biane:Yor:87}:
$
 h({\mathfrak m}) \stackrel{law} = \mu({\mathfrak m}).
$
Based on Proposition \ref{supmean}, this observation allows us to infer the exact small deviation asymptotics for the functional $h({\mathfrak m}) = \int_0^1 ds / {\mathfrak m} (s).$

\section{Acknowledgements}

The research of the authors was supported by the Federal Grant-in-Aid Program ``Human Capital for Science and Education in Innovative Russia'' (grant No. 2010-1.1-111-128-033), the Russian Foundation for Basic Research (grant No. 10-01-00154), and the Program for Supporting Leading Scientific Schools (grant No. NSh-4472.2010.1).
The authors are thankful to Professors A.~N.~Borodin, M.~A.~Lifshits and A.~I.~Nazarov for their valuable comments and suggestions.


\begin{thebibliography}{99}

\bibitem{Abr:1979}
Abramowitz,~M. and Stegun,~I.~A. (1970).
{\it Handbook of Mathematical Functions.}
New York: Dover.

\bibitem{Adle:1990}
Adler,~R.~J. (1990).
An introduction to continuity, extrema, and related topics for general Gaussian processes.
{\it IMS Lecture Notes Monogr. Ser.}
{\bf 12.}
Hayward, CA: Inst. Math. Statist.
\MR{1088478}

\bibitem{Alili:97}
Alili,~L. (1997).
On some hyperbolic principal values of Brownian local times.
In {\it Exponential Functionals and Principal Values related to Brownian Motion, Bibl. Rev. Mat. Iberoamericana}
131--154.
Madrid: Rev. Mat. Iberoamericana.
\MR{1648658}

\bibitem{Alili:Yor:97}
Alili,~L., Donati-Martin,~C. and Yor,~M. (1997).
Une identit\'{e} en loi remarquable pour l'excursion brownienne normalis\'{e}e.
In {\it Exponential Functionals and Principal Values related to Brownian Motion, Bibl. Rev. Mat. Iberoamericana}
155 --180.
Madrid: Rev. Mat. Iberoamericana.
\MR{1648659}

\bibitem{Anderson:1952}
Anderson,~T.~W. and Darling,~D.~A. (1952).
Asymptotic theory of certain ``goodness of fit'' criteria based on stochastic  processes.
{\it Ann. Math. Statist.}
{\bf 23}
193--212.
\MR{0050238}

\bibitem{Aurz:Ibra:Lifs:vZan:2008}
Aurzada~F., Ibragimov,~I.~A, Lifshits~M.~A. and van Zanten,~J.~H. (2009).
Small deviations of smooth stationary Gaussian processes.
{\it Theory Probab. Appl.}
{\bf 53}
697--707.

\bibitem{Bate:Erd:1974}
Bateman,~H. and Erdelyi~A. (1953).
{\it Higher transcendental functions.} Vol.~II.
New York: McGrawHill.

\bibitem{Begh:Niki:Orsi:2003}
Beghin,~L., Nikitin,~Ya. and Orsingher,~E. (2003).
Exact small ball constants for some Gaussian processes under $L_2$-norm.
{\it Zap. Nauchn. Sem. S.-Peterburg. Otdel. Mat. Inst. Steklov. (POMI)}
{\bf 298}
5--21.

\bibitem{Ber:Pit:Ru:1999}
Bertoin,~J., Pitman,~J. and Ruiz de Chavez,~J. (1999).
Constructions of a Brownian path with a given minimum.
{\it Electron. Commun. Probab.}
{\bf 4}
31--37.
\MR{1703609}

\bibitem{Biane:Yor:2001}
Biane,~Ph., Pitman,~J.\ and Yor,~M. (2001).
Probability laws related to the Jacobi theta and Riemann zeta functions, and Brownian excursions.
{\it Bull. Amer. Math. Soc. (N.S.)}
{\bf 38}
435--465.
\MR{1848256}

\bibitem{Biane:Yor:87}
Biane,~Ph. and Yor,~M. (1987).
Valeurs principales associ\'{e}es aux temps locaux browniens.
{\it Bull. Sci. Math. (2)}
{\bf 111}
23--101.
\MR{0886959}

\bibitem{Boga:1997}
Bogachev,~V.~I. (1998).
Gaussian measures.
{\it Math. Surveys Monogr.}
{\bf 62.}
Providence, RI: Amer. Math. Soc.
\MR{1642391}

\bibitem{Bogo:1954}
Bogolyubov,~N.~N. (1954).
On representation of Green-Schwinger functions by means of functional integrals. (Russian)
{\it Dokl. Akad. Nauk SSSR}
{\bf 99}
225--226.
\MR{0067768}

\bibitem{Boro:89}
Borodin,~A.~N. (1989).
Brownian local time.
{\it Russian Math. Surveys}
{\bf 44}(2)
1--51.
\MR{0998360}

\bibitem{Boro:Salm:2000}
Borodin, A. and Salminen, P. (2002).
{\it Handbook on Brownian Motion---Facts and Formulae.}
2nd edition.
Basel: Birkh\"auser.
\MR{1912205}

\bibitem{Ciesielski:62}
Ciesielski,~Z. and Taylor,~S.~J. (1962).
First passage times and sojourn times for Brownian motion in space and the exact Hausdorff measure of the sample path.
{\it Trans. Amer. Math. Soc.}
{\bf 103}
434--450.
\MR{0143257}

\bibitem{Cod:Lev:1955}
Coddington,~E.~A. and Levinson,~N. (1955).
{\it Theory of ordinary differential equations.}
NewYork: McGrawHill.
\MR{0069338}

\bibitem{Cs:Shi:Yor:1999}
Cs\"org\H{o},~M., Shi,~Z. and Yor,~M. (1999).
Some asymptotic properties of the local time of the uniform empirical process.
{\it Bernoulli}
{\bf 5}
1035--1058.
\MR{1735784}

\bibitem{Dudl:Hoff:Shep:1979}
Dudley,~R.~M., Hoffmann-J\o rgensen,~J. and Shepp,~L.~A. (1979).
On the lower tail of Gaussian seminorms.
{\it Ann. Probab.}
{\bf 7}
319--342.
\MR{0525057}

\bibitem{DLL}
Dunker,~T., Lifshits,~M.~A. and Linde,~W. (1998).
Small deviations of sums of independent variables.
In {\it High dimensional probability (Oberwolfach, 1996).}
{\it Progr. Probab.}
{\bf 43}
59--74.
Basel: Birkh\"auser.
\MR{1652320}

\bibitem{Fata:2003}
Fatalov,~V.~R. (2003).
Constants in the asymptotics of small deviation probabilities for Gaussian processes and fields.
{\it Russian Math. Surveys}
{\bf 58}
725--772.
\MR{2042263}

\bibitem{Fata:2007}
Fatalov,~V.~R. (2007).
Occupation times and exact asymptotics of small deviations of Bessel processes for $L_p$-norms with $p>0$.
{\it Izv. Math.}
{\bf 71}
721--752.
\MR{2360007}

\bibitem{Fatalov:2008}
Fatalov,~V.~R. (2008).
Some asymptotic formulas for the Bogoliubov Gaussian measure.
{\it Theoret. and Math. Phys.}
{\bf 157}
1606--1625.
\MR{2493782}

\bibitem{Fata:2008}
Fatalov,~V.~R. (2008).
Exact asymptotics of small deviations for a stationary Ornstein--Uhlenbeck process and some Gaussian diffusion processes in the $L_{p}$-norm, $2\le p\le\infty$.
{\it Probl. Inf. Transm.}
{\bf 44}
138--155.
\MR{2435241}

\bibitem{Fedo:2009}
Fedoryuk,~M.~V. (1993).
{\it Asymptotic analysis: linear ordinary differential equations.}
Berlin: Springer.
\MR{1295032}

\bibitem{Ferr:Vieu:2006}
Ferraty,~F. and Vieu,~Ph. (2006).
{\it Nonparametric functional data analysis. Theory and practice.}
New York: Springer.
\MR{2229687}

\bibitem{From:From:1967}
Fr\"oman,~N. and Fr\"oman,~P.~O. (1965).
{\it JWKB Approximation. Contributions to the Theory.}
Amsterdam: North-Holland.
\MR{0173481}

\bibitem{Gao:Hann:Lee:Torc:2003}
Gao,~F., Hannig,~J., Lee,~T.-Y. and Torcaso,~F. (2003).
Laplace transforms via Hadamard factorization.
{\it Electron. J. Probab.}
{\bf 8}(13)
1--20.
\MR{1998764}

\bibitem{Gao:Hann:Lee:Torc:2004}
Gao,~F., Hannig,~J., Lee,~T.-Y. and Torcaso,~F. (2004).
Exact $L_2$-small balls of Gaussian processes.
{\it J. Theoret. Probab.}
{\bf 17}
503--520.
\MR{2053714}

\bibitem{Gi:1957}
Gikhman,~I.~I. (1957).
On a nonparametric criterion of homogeneity for $k$ samples.
{\it Theory Probab. Appl.}
{\bf 2}
369--373.

\bibitem{Gutierrez:1987}
Guti\'errez Jaimez,~R. and Valderrama Bonnet,~M.~J. (1987).
On the Karhunen-Lo\`{e}ve expansion for transformed processes.
{\it Trabajos Estad\'\i st.}
{\bf 2}(2)
81--90.

\bibitem{Ibra:1979}
Ibragimov,~I.~A. (1982).
Hitting probability of a Gaussian vector with values in a Hilbert space in a sphere of small radius.
{\it J. Sov. Math.}
{\bf 20}
2164--2175.

\bibitem{Ito:McKean:1965}
It\^o,~K. and McKean,~H.~P. (1965).
{\it Diffusion processes and their sample paths.}
Berlin: Springer.
\MR{0199891}

\bibitem{Jan:2007}
Janson,~S. (2007).
Brownian excursion area, Wright's constants in graph enumeration, and other Brownian areas.
{\it Probab. Surv.}
{\bf 4}
80--145.
\MR{2318402}

\bibitem{Kennedy:76}
Kennedy,~D.~P. (1976).
The distribution of the maximum Brownian excursion.
{\it J. Appl. Probab.}
{\bf 13}
371--376.
\MR{0402955}

\bibitem{Niki:Khar:2004}
Kharinski,~P.~A. and Nikitin,~Ya.~Yu. (2006).
Sharp small deviation asymptotics in $L_2$-norm for a class of Gaussian processes.
{\it J. Math. Sci. (N.Y.)}
{\bf 133}
1328-–1332.

\bibitem{Kie:1959}
Kiefer,~J. (1959).
$K$-sample analogues of the Kolmogorov--Smirnov and Cram\'{e}r--V.~Mises tests.
{\it Ann. Math. Statist.}
{\bf 30}
420-447.
\MR{0102882}

\bibitem{Levy:1965}
L\'{e}vy,~P. (1965).
{\it Processus stochastiques et mouvement brownien.}
Paris: Gauthier-Villars.
\MR{0190953}

\bibitem{Li:1992}
Li,~W.~V. (1992).
Comparison results for the lower tail of Gaussian seminorms.
{\it J. Theoret. Probab.}
{\bf 5}
1--31.
\MR{1144725}

\bibitem{Li:Shao:2001}
Li,~W.~V. and Shao,~Q.-M. (2001).
Gaussian processes: inequalities, small ball probabilities and applications.
In {\it Stochastic Processes: Theory and Methods.}
{\it Handbook of Statist.}
{\bf 19}
533--597.
Amsterdam: North-Holland.
\MR{1861734}

\bibitem{Lifs:1995}
Lifshits,~M.~A. (1995).
{\it Gaussian Random Functions.}
Dordrecht: Kluwer.
\MR{1472736}

\bibitem{Lifs:1997}
Lifshits,~M.~A. (1997).
On the lower tail probabilities of some random series.
{\it Ann. Probab.}
{\bf 25}
424--442.
\MR{1428515}

\bibitem{Lifs:1999}
Lifshits,~M.~A. (1999).
Asymptotic behavior of small ball probabilities.
In {\it Probability Theory and Mathematical Statistics: Proceedings of the Seventh International Vilnius Conference.}
453--468.
Vilnius: TEV.

\bibitem{Lifs:2010}
Lifshits,~M.~A.
Bibliography on small deviation probabilities. Available at \newline  \url{http://www.proba.jussieu.fr/pageperso/smalldev/biblio.pdf}.

\bibitem{Mansuy:2005}
Mansuy,~R. (2005).
An interpretation and some generalizations of the Anderson–Darling statistics in terms of squared Bessel bridges.
{\it Statist. Probab. Lett.}
{\bf 72}
171-–177.
\MR{2137123}

\bibitem{Man:Yor:2008}
Mansuy,~R. and Yor~M. (2008).
{\it Aspects of Brownian motion.}
Berlin: Springer-Verlag.
\MR{2454984}

\bibitem{Mi:1994}
Mikhailova,~E.~M. (1994).
Asymptotic distributions for Brownian motion with drift.
{\it Russian Math. Surveys}
{\bf 49}(4)
171--173.
\MR{1309456}

\bibitem{Naza:2003}
Nazarov,~A.~I. (2003).
On the sharp constant in the small ball asymptotics of some Gaussian processes under $L_2$-norm,
{\it J. Math. Sci. (N.Y.)}
{\bf 117}
4185--4210.

\bibitem{Naza:2009}
Nazarov,~A.~I. (2009).
Exact $L_2$-small ball asymptotics of Gaussian processes and the spectrum of boundary-value problems.
{\it J. Theoret. Probab.}
{\bf 22}
640--665.
\MR{2530107}

\bibitem{Naza:Niki:2004}
Nazarov,~A.~I. and Nikitin,~Ya.~Yu. (2004).
Exact $L_2$-small ball behavior of integrated Gaussian processes and spectral asymptotics of boundary value problems.
{\it Probab. Theory Related Fields}
{\bf 129}
469--494.
\MR{2078979}

\bibitem{Naza:Puse:2009}
Nazarov,~A.~I. and Pusev,~R.~S. (2009).
Exact small deviation asymptotics in $L_2$-norm for some weighted Gaussian processes.
{\it J. Math. Sci. (N.Y.)}
{\bf 163}
409–-429.

\bibitem{Niki:Orsi:2004}
Nikitin,~Ya.~Yu. and Orsingher,~E. (2006).
Exact small deviation asymptotics for the Slepyan and Watson processes.
{\it J. Math. Sci. (N.Y.)}
{\bf 137}
4555–-4560.

\bibitem{Pitm:1975}
Pitman,~J. (1975).
One-dimensional Brownian motion and the three-dimensional Bessel process.
{\it Adv. Appl. Prob.}
{\bf 7}
511--526.
\MR{0375485}

\bibitem{Pitman:Yor:96}
Pitman,~J. and Yor,~M. (1996).
Quelques identit\'{e}s en loi pour les processus de Bessel.
In {\it Hommage \`{a} P.~A.~Meyer et J.~Neveu.}
{\it Ast\'{e}risque}
{\bf 236}
249-–276.

\bibitem{Pitman:Yor:99}
Pitman,~J. and Yor,~M. (1999).
The law of the maximum of a Bessel bridge.
{\it Electron. J. Probab.}
{\bf 4}(15)
1--35.
\MR{1701890}

\bibitem{Puse:2010}
Pusev,~R.~S. (2010).
Asymptotics of small deviations of the Bogoliubov processes with respect to a quadratic norm.
{\it Theoret. Math. Phys.}
{\bf 165}
1349--1358.

\bibitem{Rev:Yor:2001}
Revuz,~D. and Yor,~M. (2001).
{\it Continuous martingales and Brownian motion.}
Berlin: Springer.

\bibitem{Roge:Will:2000}
Rogers,~L.~C.~G. and Williams,~D. (2000).
{\it Diffusions, Markov processes and martingales. Vol.~2. It\^o calculus.}
Cambridge: Cambridge University Press.
\MR{1780932}

\bibitem{Sank:1999}
Sankovich,~D.~P. (1999).
Gaussian functional integrals and Gibbs equilibrium averages.
{\it Theoret. Math. Phys.}
{\bf 119}
670--675.
\MR{1718649}

\bibitem{Sank:2001a}
Sankovich,~D.~P. (2001).
Some properties of functional integrals with respect to the Bogoliubov measure.
{\it Theoret. Math. Phys.}
{\bf 126}
121--135.
\MR{1858200}

\bibitem{Sank:2001b}
Sankovich,~D.~P. (2001).
Metric properties of Bogoliubov trajectories in statistical equilibrium theory.
{\it Theoret. Math. Phys.}
{\bf 127}
513--527.
\MR{1863523}

\bibitem{Sank:2005}
Sankovich,~D.~P. (2005).
The Bogolyubov functional integral.
{\it Proc. Steklov Inst. Math.}
{\bf 251}
213--245
\MR{2234384}

\bibitem{Sh:We:1986}
Shorack,~G. and Wellner,~J. (1986).
{\it Empirical processes with applications to statistics.}
New York: Wiley.
\MR{0838963}

\bibitem{Syta:1974}
Sytaya,~G.~N. (1974).
On some asymptotic representations of the Gaussian measure in a Hilbert space.
(Russian)
{\it Theory Stoch. Process.}
{\bf 2}
93--104.

\bibitem{Titch:80}
Titchmarsh,~E.~C. (1939).
{\it The theory of functions.} 2nd edition.
London: Oxford University Press.

\bibitem{vdVa:vZan:2008}
van~der~Vaart,~A.~W. and van~Zanten,~J.~H. (2008).
Rates of contraction of posterior distributions based on Gaussian process priors.
{\it Ann. Statist.}
{\bf 36}
1435--1463.
\MR{2418663}

\bibitem{Will:1970}
Williams,~D. (1970).
Decomposing the Brownian path.
{\it Bull. Amer. Math. Soc.}
{\bf 76}
871--873.
\MR{0258130}

\bibitem{Zolo:1979}
Zolotarev,~V.~M. (1979).
Gaussian measure asymptotics in $l_2$ on a set of centered spheres with radii tending to zero.
{\it 12th Europ. Meeting of Statisticians, Varna.}
254.

\end{thebibliography}
\end{document}